\documentclass[review=false,onefignum,onetabnum]{siamart190516}
\usepackage{amsmath} 

\usepackage{amssymb}
\usepackage[disable]{todonotes}
\usepackage[algo2e,vlined,ruled]{algorithm2e}
\usepackage{bm}
\usepackage{multirow}
\usepackage{graphicx}
\usepackage{subfigure}
\usepackage{algorithmic}
\usepackage{amsmath}
\usepackage[normalem]{ulem}

\newtheorem{assumption}[theorem]{Assumption}

\usepackage[numbers]{natbib}

\DeclareMathOperator{\Tr}{Tr}
\renewcommand{\theequation}{\arabic{section}.\arabic{equation}}
\newcommand{\be}{\begin{equation}}
\newcommand{\ee}{\end{equation}}
\newcommand{\ba}{\begin{array}}
\newcommand{\ea}{\end{array}}
\newcommand{\bea}{\begin{eqnarray}}
\newcommand{\eea}{\end{eqnarray}}
\newcommand{\beas}{\begin{eqnarray*}}
\newcommand{\eeas}{\end{eqnarray*}}

\newtheorem{remark}{Remark}[section]

\newcommand{\rn}[1]{\uppercase\expandafter{\romannumeral #1}}

\newcommand{\tabincell}[2]{\begin{tabular}{@{}#1@{}}#2\end{tabular}} 

\title{On the Analysis of Model-free Methods for the Linear Quadratic Regulator}

\begin{document}

\title{On the Analysis of Model-free Methods for the Linear Quadratic Regulator}

\author{Zeyu Jin\thanks{School of Mathematical Sciences, Peking University, CHINA (\email{1801210096@pku.edu.cn}).}
\and Johann Michael Schmitt\thanks{Beijing International Center for Mathematical
Research, Peking University, CHINA (\email{MichelSchmitt2@web.de}). }
\and Zaiwen Wen\thanks{Beijing International Center for Mathematical
		Research, Peking University, CHINA (\email{wenzw@pku.edu.cn}).
		Research supported in part by the NSFC grant 11831002 and  Beijing Academy  of Artificial Intelligence.} 
 }

%
\headers{On the Analysis of Model-free Methods for the LQR}{Z. Jin, J.M. Schmitt,
and Z. Wen}
\maketitle

\begin{abstract}
Many reinforcement learning methods achieve great success in practice but lack theoretical foundation. In this paper, we study the convergence analysis on the problem of the Linear  Quadratic Regulator (LQR). The global linear convergence properties and sample complexities are established for several popular algorithms such as the policy gradient algorithm, TD-learning and the actor-critic (AC) algorithm. Our results show that the actor-critic algorithm can reduce the sample complexity compared with the policy gradient algorithm. Although our analysis is still preliminary, it explains the benefit of AC algorithm in a certain sense.
\end{abstract}


\section{Introduction}
Reinforcement learning (RL) involves training an agent such that the agent takes a sequence of actions to minimize its cumulative cost (or maximize its cumulative reward), 
 see \cite{sutton2018reinforcement} for a general introduction to RL. 
Model free methods which do not estimate and use the transition kernel directly in RL enjoy wide popularity.
They have achieved great success in many fields, such as robotics \cite{kober2013reinforcement}, biology \cite{mahmud2018applications}, competitive gaming \cite{mnih2015human} and so on. 
In order to improve the performance of these algorithms, a theoretical understanding of questions regarding global convergence and sample complexity becomes more and more crucial. However, since RL problems are non-convex, it is even hard to prove convergence for model free methods. 

The drawback of the non-convexity also appears in the Linear Quadratic Regulator (LQR) problem which is an elementary problem in system control and well understood. Therefore, we consider LQRs as a first step. In practice, people often estimate the transition matrices directly (called system identification) and then design a linear policy for the LQR.  Since model-free methods become more and more popular, there is a large number of literatures analyzing model free methods for LQRs, see for example  \cite{tu2017least,Fazel2018Global,Dean2018,malik2018derivative}. The authors of \cite{tu2017least} analyze the sample complexity of the least-squares temporal difference (LSTD) method for one fixed policy in the LQR setting. There are also contributions \cite{Fazel2018Global,malik2018derivative} in which the properties of the cumulative cost with respect to the policy are analyzed and global convergence of policy iterations generated by some zero-order optimization method is shown. 
In the analysis of LQRs, we can without loss of generality (w.l.o.g.) restrict ourselves to linear policies, since it can be shown that optimal policies are linear.
Although the  cumulative cost in the LQR problem is non-convex with respect to
the policy, any locally optimal policy is globally optimal. In this paper, we
analyze some basic model free methods at the example of the LQR setting and
derive the sample complexity, i.e., the number of samples which is at least required to guarantee convergence up to some specified tolerance.

\subsection{Related work}
TD-learning \cite{sutton1988learning} and Q-learning \cite{watkins1992q} are basic and popular value based methods. There is a line of work examining the convergence of TD-learning and Q-learning with linear value function approximations for Markov decision processes (MDPs), see, e.g., \cite{tsitsiklis1997analysis,borkar2000ode,bhandari2018finite,zou2019finite}. The convergence with probability $1$ is proved in \cite{tsitsiklis1997analysis,borkar2000ode}. Moreover, a non-asymptotic analysis of TD-learning and Q-learning is provided in \cite{bhandari2018finite} and \cite{zou2019finite}, respectively. The authors in \cite{bhatnagar2009convergent} extend the asymptotic analysis of TD-learning to the case of nonlinear value function approximations. In addition, there is a large number of contributions analyzing least-squares TD \cite{bradtke1996linear,lazaric2010finite} and gradient TD \cite{sutton2009fast,liu2015finite} which also require linear value function approximations. 

In policy based methods, the policy gradient method \cite{williams1992simple,Sutton1999Policy} and the actor-critic method \cite{Sutton1999Policy,konda2000actor} achieve empirical success. Sutton et al. show in \cite{Sutton1999Policy} an asymptotic convergence result for the actor-critic method applied to MDPs under a compatibility condition on the value function approximations. In \cite{dai2017sbeed, zhang2019global},  a non-asymptotic analysis of actor-critic methods is provided under some parametrization assumptions on the MDP requiring that the state space is finite. In \cite{Fazel2018Global,malik2018derivative}, the policy gradient method with evolution strategy is used for the LQR setting and a global convergence result is derived. The sample complexity of the LSTD method for LQRs is shown to be $\Omega(n^2/\epsilon^2)$ for any given policy in \cite{tu2017least}. There is also a line of work considering the design and analysis of model based methods for LQRs, see for example \cite{abbasi2011regret,dean2017sample,Dean2018}.  

\subsection{Contribution}
In this paper, our motivation is to analyze general RL methods in the LQR setting. Unlike the MDP,  the state space and the action space in the LQR  are both infinite dimensional which makes LQR problems difficult to handle. On the other hand, the LQR represents a simple but also typical continuous problem in RL. 
First, we  use the TD-learning approach with linear value function approximations for the policy evaluation of LQRs, which is a quite popular method used for general RL problems. Compared to \cite{tu2017least}, we focus on the sample complexity of TD-learning instead of the LSTD approach in which linear equations need to be solved. 
In addition, we also analyze the global convergence of policy iterations generated by TD-learning, which is inspired by the work of \cite{Fazel2018Global}.

Instead of using evolution strategies as in \cite{Fazel2018Global,malik2018derivative}, we prove global linear convergence of the policy gradient method and the actor-critic algorithm in the LQR setting, which are much more widely used methods in practice.
Another difference is that we focus on the more complex noise cases instead of initial random cases.
We also show that the policy gradient is equivalent to the gradient of the cumulative cost with respect to the policy parameters. Moreover, our work combines the analysis of the value function and the analysis of the policy. 

The estimation of the complexities of the policy iterations generated by TD-learning, the gradient method and the actor-critic method are given
in Table \ref{tab:complexity}. These complexities are based on the results of this paper, in particular on Theorem~\ref{thm:policy_iteration_convergence}, Theorem~\ref{PGTH} and Theorem~\ref{thm:policy_ac_convergence}.   In Table \ref{tab:complexity}, $\gamma$ is the discount factor and $\epsilon$ is the error tolerance with respect to the globally optimal value.

\begin{table}[htp]\label{tab:complexity}
	
\begin{tabular}{|c|c|c|c|}
\hline 
Algorithm & TD steps&\tabincell{c}{Length of trajectory\\for policy gradient}& Iterations\\
\hline  
Policy Iteration with TD & $O\left(\frac{1}{\epsilon(1-\gamma)^5}\right)$ &/& $O(1)$\\
\hline 
Policy Gradient &/&$O\left(\frac{1}{(1-\gamma)^2}\right)$&$O\left(\frac{1}{\delta\epsilon(1-\gamma)^7}\right)$\\
\hline 
Actor-Critic Algorithm & $O\left(\frac{1}{\delta\epsilon(1-\gamma)^4}\right)$&$O\left(\frac{1}{(1-\gamma)^4}\right)$&$O\left(\frac{1}{\delta\epsilon(1-\gamma)}\right)$\\
\hline 
\end{tabular}
\caption{A summary of complexity analysis}
\end{table}
 
\subsection{Organization}
This paper is organized as follows. Section~\ref{sect:intro} starts with a
description of the general LQR problem. In order to solve LQRs in the RL
framework, we introduce the policy function and convert the original problem into a policy optimization problem. Section~\ref{sect:PE} discusses the TD-learning approach with linear value function approximations and the convergence of this method. In Section~\ref{sect:PI}, the process of policy iterations is combined with TD-learning and the convergence of this method is analyzed. Sections~\ref{sect:PG} and \ref{sect:AC} present convergence results of the policy gradient method and the actor critic algorithm. In Section~\ref{sect:conclusion}, we finally give a brief summary of the 
main results of this paper.


%
\section{Preliminary}
\label{sect:intro}
\setcounter{equation}{0}

Linear time invariant (LTI) systems have the following form:
\[x_{t+1}=Ax_t+Bu_t+\omega_t,\]
where $A\in \mathbb{R}^{n\times n}$, $B\in \mathbb{R}^{n\times d}$ and the sequence $\omega_t$ describes unbiased, independent and identically distributed noise. We call $x_t\in \mathbb{R}^n$ state and
$u_t\in \mathbb{R}^d$ control. The control is measured by the cost function 
\[c_t=x_t^\top Sx_t+u_t^\top Ru_t\]
with positive definite matrices $S\in\mathbb{R}^{n\times n}$ and
$R\in\mathbb{R}^{d\times d}$.
Thus, the cumulative cost with discount factor $\gamma\in (0,1)$ is given by
$ \sum_{t=0}^{\infty}\gamma^t(x_t^\top  Sx_t+u_t^\top Ru_t)$.
The LQR problem consists in finding the control which 
minimizes the expectation of the cumulative cost leading to 
the following optimization problem which we will refer to as the 
LQR problem
\begin{equation}
\label{originalLQRproblem}
\begin{aligned}
\min_{\{u_t\}_{t=0}^\infty}&\quad \sum_{t=0}^{\infty}\mathbb{E}[\gamma^t(x_t^\top Sx_t+u_t^\top Ru_t)]\\
\text{s.t.}&\quad x_0=0,\quad x_{t+1}=Ax_t+Bu_t+\omega_t.\\
\end{aligned}
\end{equation}

Optimal control theory shows that the optimal control input can be written as a linear function with respect to the state, i.e., 
$u_t=K^*x_t$,
where $K^*\in \mathbb{R}^{d\times n}$. The optimal control gain is given by
\be
\label{kstar}
K^*=-\gamma(R+\gamma B^\top P_{\gamma}B)^{-1}B^\top P_{\gamma}A,
\ee
where $P_{\gamma}$ is the solution of the algebraic Riccati equation (ARE)
\[
P_{\gamma}=\gamma A^\top P_{\gamma}A+S-\gamma^2 A^\top P_{\gamma}B(\gamma B^\top P_{\gamma}B+R)^{-1}B^\top P_{\gamma}A.
\]
Hence, the optimal policy $K^*$  only depends on $A$, $B$, $S$, $R$ and $\gamma$ and
the optimal policy formula above is shown in \cite{Fazel2018Global}. We will explain the optimality of this policy in Section~\ref{sect:PI}.

\subsection{Stochastic Policies}
In practice, the system is not known exactly. Therefore, it is
popular to use model-free methods to solve the LQR problem 
\eqref{originalLQRproblem}. We also want to investigate the 
properties of those model-free methods for LQRs. 
We observe that the controls $u_t$ form $\mathbb{R}^d$-valued 
sequences $\{u_t\}_{t=0}^{\infty}$ belonging to an infinite-dimensional
vector space. Thus, the analysis of the LQR problem 
\eqref{originalLQRproblem} has to be carried out with care. Since the  
problem \eqref{originalLQRproblem} can be viewed as a Markov decision 
process, people often use policy functions, which are maps from 
the state space to the control space, to represent the control. In
this paper, the policies are constrained to some specific set of
policy functions to simplify the problem. 

In general, we can employ  
Gaussian policies

\[
u_t\sim \pi_{\theta}(\cdot|x_t)=N(f_{\theta}(x_t),\sigma^2I_d),
\]
where $\theta$ is the parameter of the policy and $\sigma>0$ is a 
fixed constant. There are several possibilities to choose function spaces for 
$f_{\theta}$, such as linear function spaces and neural networks. 
As explained above, optimal controls of \eqref{originalLQRproblem} 
depend linearly on the state. Therefore,  since the usage of
nonlinear policy functions considerably complicates the analysis,
we focus on linear policy functions, which can be represented
by matrices $K\in \mathbb{R}^{d\times n}$ as follows 
\[
 u_t\sim \pi_{K}(\cdot|x_t)=N(Kx_t,\sigma^2I_d).
\]
This yields an equivalent form of the control
\[u_t=Kx_t+\eta_t,\quad \eta_t\sim N(0,\sigma^2I_d).\] 
In addition, the probability density function of $\pi_K$ has the explicit form
\begin{align*}
\pi_K(u|x)=\frac{1}{(2\pi)^{d/2}\sigma^d}\mathrm{exp}\left(-\frac{\|u-Kx\|_2^2}{2\sigma^2}\right).
\end{align*}

The advantage of stochastic policies is that the policy gradient method can be applied to the problem, which we will discuss in a later section. In addition, stochastic policies are benefitial to the exploration, which is quite important in reinforcement learning.  

Under the policy $\pi_K$, the dynamic system can be written as  
\[x_{t+1}=Ax_t+Bu_t+\omega_t=(A+BK)x_t+B\eta_t+\omega_t.\] 
Let $\widetilde\omega_t := B\eta_t+\omega_t$, $D_{\widetilde \omega}=\mathbb{E}[\widetilde\omega_t\widetilde\omega_t^\top]$, $D_{ \omega}=\mathbb{E}[\omega_t\omega_t^\top]$ and   
$\mathbb{E}[\eta_t\eta_t^\top]=\sigma^2I_d$.
Using these abbreviations, the LQR problem \eqref{originalLQRproblem} 
can be reformulated as
\be
\label{tck}
\begin{aligned}
\min_K&\;\; J(K)=\sum_{t=0}^{\infty}\mathbb{E}\gamma^t[x_t^\top (S+K^\top RK)x_t]+\frac{\sigma^2\Tr[R]}{1-\gamma} \\
\text{s.t.}&\;\;x_0=0,\;\;x_{t+1}=(A+BK)x_t+\widetilde\omega_t,\\
\end{aligned}
\ee
where we have used the linearity of the trace operator, more precisely
\begin{align*}
 \sum_{t=0}^{\infty}\mathbb{E}[\gamma^t \eta_t^\top R \eta_t]
 &=\sum_{t=0}^{\infty}\gamma^t \mathbb{E}\Tr [\eta_t^\top R \eta_t]
 =\sum_{t=0}^{\infty}\gamma^t \mathbb{E}\Tr [\eta_t\eta_t^\top R]
 =\sum_{t=0}^{\infty}\gamma^t \Tr[\mathbb{E}[\eta_t\eta_t^\top] R]\\
 &=\frac{\sigma^2\Tr[R]}{1-\gamma}.
\end{align*}
We will use this trick in several places of the rest of this paper.

Let the domain of feasible policies be given by
\textbf{Dom}$:=\{K\;|\;\rho(A+BK)<1\}$, where $\rho(X)$ denotes the 
spectral radius of the matrix $X$. In the following, we assume that \textbf{Dom} is non-empty.  Policies lying in the feasible domain are called stable. The set \textbf{Dom} is 
non-convex, see \cite{Fazel2018Global} and for all $K\in\bf{Dom}$ the corresponding value $J(K)$ is finite. By Proposition 3.2 in \cite{tu2017least}, for any stable $K$
and any $\rho\in (\rho(A+BK),1)$, there exists a $\Gamma_K>0$, such that for any $k\geq 1$, the following inequality holds 
\be\label{eq:stability}
\|(A+BK)^k\|_2\leq\Gamma_K \rho^k.
\ee
For any compact subset $\mathcal K\subset$\textbf{Dom}, we 
can find uniform constants $\Gamma$ and $\rho$ such that
\eqref{eq:stability} holds for all $K\in\mathcal K$. Unfortunately, since $0<\gamma<1$, $J(K)$ may also be finite for unstable $K$. More precisely, one can show that $J(K)$ is finite if and only if
\be\label{eq:dom_ext}
\rho(A+BK)<\frac{1}{\sqrt{\gamma}}.
\ee
We observe that the constraints $x_{t+1}=(A+BK)x_t+\widetilde \omega_t$
in \eqref{tck} yield for $t\geq 1$
\be\label{eq:state}
x_t=\sum\limits_{i=0}^{t-1}(A+BK)^{t-i-1}\widetilde \omega_i.
\ee
Provided that $K$ satisfies \eqref{eq:dom_ext}, we can insert
\eqref{eq:state} into $J(K)$. Exploiting that $\mathbb E[\tilde w_i\tilde w_j^\top]=0$
for $j\neq i$, we obtain the following
analytic form of the objective function
\be\label{redcost}
 J(K)=\frac{\gamma}{1-\gamma}\sum_{i=0}^{\infty}\Tr[D_{\widetilde\omega}\gamma^i[(A+BK)^i]^\top (S+K^\top RK)(A+BK)^{i}]+\frac{\sigma^2\Tr[R]}{1-\gamma}.
\ee
Let $P_{K,\gamma}:=\sum_{i=0}^{\infty}\gamma^i[(A+BK)^i]^\top (S+K^\top RK)(A+BK)^{i}$. Then $ J(K)$ can be further simplified to
\be\label{eq:cost_function_simplified}
 J(K)=\frac{\gamma}{1-\gamma}\Tr[D_{\widetilde\omega}P_{K,\gamma}]+\frac{\sigma^2\Tr[R]}{1-\gamma}.
\ee
In the next result, we observe that for $\gamma$ which are 
sufficiently close to $1$, we can restrict our search for 
the optimal policy to the set of stable policies.
\begin{lemma}\label{lem:restriction_to_dom}
Suppose that $K$ is a stable policy. Then for 
any $\nu\ge1$ there exists a discount factor $0<\gamma<1$
and a positive constant $\beta=\beta(K,\nu)<1$ 
such that 
\be\label{eq:lemma_restriction_to_dom}
J(\bar K)-J(K^*)> \nu(J(K)-J(K^*))
\ee 
is valid for any policy $\bar K$ with $\rho(A+B\bar{K})> 1-\beta$, where $K^*$ denotes an optimal policy. In addition, the set 
\be\label{eq:lemma_restriction_to_dom_set}
\text{\bf{Dom}}_{K,\nu}=\left\{\tilde K: J(\tilde K)-J(K^*)\leq \nu(J(K)-J(K^*))\right\}
\ee 
is a compact subset of $\bf{Dom}$. Finally, there exist constants $\rho\in (1-\beta,1)$ and $\Gamma_{K,\nu}$ such that 
\be\label{eq:stability2}
\|(A+B\tilde K)^k\|_2\leq\Gamma_{K,\nu} \rho^k.
\ee
holds for all $\tilde K\in\text{\bf{Dom}}_{K,\nu}$.
\end{lemma}
\begin{proof}
Let $\bar K$ be an arbitrary policy with $\rho(A+B\bar{K})\geq 1-\beta$, where $\beta$ will be determined later. We can w.l.o.g. assume that $\bar K$ satisfies \eqref{eq:dom_ext}. Otherwise, $J(\bar K)$ would be infinite and \eqref{eq:lemma_restriction_to_dom} holds trivially. First, we observe that
\eqref{eq:lemma_restriction_to_dom} follows by
\be\label{eq:lemma_restriction_to_dom_eq}
\Tr[D_{\widetilde\omega}P_{\bar K,\gamma}]> \nu\Tr[D_{\widetilde\omega}P_{K,\gamma}]
\ee 
Next, we introduce the abbreviations $\rho_K=\rho(A+BK)$ and $\rho_{\bar K}=\rho(A+B\bar{K})$. Since $D_{\widetilde\omega}$, $S$ and $Q$ are positive definite, using \eqref{eq:stability} and the fact that $\rho(Z)\leq\|Z\|_2\leq\|Z\|_F\leq\sqrt{n}\|Z\|_2$ holds for any matrix $Z\in\mathbb{R}^{n\times n}$, we can deduce
\begin{align*}
\Tr[D_{\widetilde\omega}P_{K,\gamma}]
&\leq \sigma_{\max}(D_{\widetilde w})\Tr[P_{K,\gamma}]\\
&\leq \sigma_{\max}(D_{\widetilde w})\sigma_{\max}(S+K^\top RK)\sum_{t=0}^\infty\gamma^t\|(A+BK)^t\|_F^2\\
&\leq \frac{\sigma_{\max}(D_{\widetilde w})\sigma_{\max}(S+K^\top RK)n\Gamma^2_K}{1-\gamma\left(\frac{\rho_K+1}{2}\right)^2}
\end{align*}
where  $\Gamma_K$ denotes the constant of the policy $K$ in 
\eqref{eq:stability} for $\rho=\frac{\rho_K+1}{2}$.
By $\|(A+B\bar K)^t\|_F\ge \rho((A+B\bar K)^t)=\rho^t_{\bar K}$, we further get 
\begin{align*}
\Tr[D_{\widetilde\omega}P_{\bar K,\gamma}]
\geq \frac{\sigma_{\min}(D_{\widetilde w})\sigma_{\min}(S)}{1-\gamma\rho_{\bar K}^2}.
\end{align*}
Let  $c_u=\sigma_{\max}(D_{\widetilde w})\sigma_{\max}(S+K^\top R K)n\Gamma^2_K$ and $c_l=\sigma_{\min}(D_{\widetilde w})\sigma_{\min}(S)$.

Using these two inequalities, we observe that \eqref{eq:lemma_restriction_to_dom_eq} and therefore also \eqref{eq:lemma_restriction_to_dom} follow by
\be\label{eq:proof_lemma_restriction_to_dom}
\frac{c_u\nu}{1-\gamma\left(\frac{\rho_K+1}{2}\right)^2}
< \frac{c_l}{1-\gamma\rho_{\bar K}^2}.
\ee
We can w.l.o.g. assume that $c_u\nu>c_l$, otherwise $c_u$ and $c_l$ can be rescaled such that this holds. Hence, isolating $\gamma$ in \eqref{eq:proof_lemma_restriction_to_dom}, we conclude that
\be\label{eq:lemma_restriction_to_dom20}
0<\frac{c_u\nu-c_l}{c_u\nu\rho^2_{\bar K}-c_l\left(\frac{\rho_K+1}{2}\right)^2}< \gamma.
\ee
Next, we observe that 
\be\label{eq:lemma_restriction_to_dom21}
\begin{aligned}
\frac{c_u\nu-c_l}{c_u\nu\rho^2_{\bar K}-c_l\left(\frac{\rho_K+1}{2}\right)^2}&< \frac{c_u\nu-c_l}{c_u\nu(1-\beta)^2-c_l+c_l\left[1-\left(\frac{\rho_K+1}{2}\right)^2\right]}\\
&<\frac{c_u\nu-c_l}{c_u\nu-2c_u\nu\beta-c_l+c_l\left[1-\left(\frac{\rho_K+1}{2}\right)^2\right]}<1
\end{aligned}
\ee
is valid for constants $\beta>0$ satisfying
\be\label{eq:constant_epsilon_K}
\beta\leq\frac{c_l}{4\nu c_u}\left[1-\left(\frac{\rho_K+1}{2}\right)^2\right].
\ee
We conclude that for $\beta>0$ satisfying \eqref{eq:constant_epsilon_K} and for $\gamma$ with
\be\label{eq:lemma_restriction_to_dom2}
0<\frac{c_u\nu-c_l}{c_u\nu-c_l+\frac{c_l}{2}\left[1-\left(\frac{\rho_K+1}{2}\right)^2\right]}< \gamma<1,
\ee
the inequalities \eqref{eq:proof_lemma_restriction_to_dom}  and therefore \eqref{eq:lemma_restriction_to_dom} hold. Next, we observe that 
\begin{align*}
\text{\bf{Dom}}_{K,\nu}\subset\left\{\tilde K: \rho(A+B\tilde K)\leq 1-\beta\right\},
\end{align*}
which is obviously a compact subset of \textbf{Dom}. The last statement of the theorem follows directly by Proposition 3.2 in \cite{tu2017least}.
\end{proof}
For instance, if a stable policy $K$ is given, then we can choose a $\gamma$ close to $1$ such that the level set of $J(K)$ is compact and only has stable policies by \cref{lem:restriction_to_dom}.
 Hence this lemma is quite important in the following discussion.

For a policy $K$ satisfying \eqref{eq:dom_ext}, the objective function of \eqref{redcost} is differentiable
in a sufficiently small neighborhood of $K$. Hence, we can compute the gradient of 
\eqref{redcost} which is given by
\be
\label{exgradient2}
\nabla  J(K)=2\left((R+\gamma B^\top P_{K,\gamma}B)K+\gamma B^\top P_{K,\gamma}A\right)\Sigma_{K,\gamma},
\ee
where $\Sigma_{K,\gamma}=\sum_{t=0}^{\infty}\mathbb{E}[\gamma^tx_tx_t^\top]$ and the sequence $\{x_t\}_{t=0}^{\infty}$ is generated by policy $u_t\sim \pi_K(\cdot|x_t)$. This gradient form is obtained by Lemma 1 in \cite{Fazel2018Global} and Lemma 4 in \cite{malik2018derivative}.

From the representation of $J(K)$ in \eqref{eq:cost_function_simplified}, we can establish the following relations between $J(K)$, $P_{K,\gamma}$ and $\Sigma_{K,\gamma}$,
see also Lemma 13 in \cite{Fazel2018Global}.
\begin{lemma}
Let $K$ be a policy such that $J(K)$ is finite. Then the representation of $J(K)$ in 
\eqref{eq:cost_function_simplified} yields the following two inequalities:
\be\label{eq:J_est_1}
J(K)\ge \frac{\gamma}{1-\gamma}\sigma_{\mathrm{min}}(D_{\widetilde \omega}) \|P_{K,\gamma}\|_F\ee 
and
\be\label{eq:J_est_2}
J(K)\ge \sigma_{\mathrm{min}}(S)\|\Sigma_{K,\gamma}\|_F\ee
\end{lemma}

\begin{proof}
Since $J(K)\ge\frac{\gamma}{1-\gamma}\Tr[D_{\widetilde\omega}P_{K,\gamma}]$ and $P_{K,\gamma}\succ 0$, we obtain
\[J(K)\ge \frac{\gamma}{1-\gamma}\sigma_{\mathrm{min}}(D_{\widetilde\omega})\Tr[P_{K,\gamma}]\ge  \frac{\gamma}{1-\gamma}\sigma_{\mathrm{min}}(D_{\widetilde\omega})\|P_{K,\gamma}\|_F.\]

The second inequality is derived by \[J(K)\ge \sum_{t=0}^{\infty}\gamma^t\mathbb  E[x_t^\top Sx_t]\ge\sigma_{\min}(S)\Tr[\Sigma_{K,\gamma}]\ge \sigma_{\min}(S)\|\Sigma_{K,\gamma}\|_F.\]
\end{proof}

Next, we will gather some useful properties of \eqref{tck}, which will serve
as tools for the convergence analysis of the methods introduced in the 
following sections. Moreover, we will also have a look at the difficulties
showing up in the theoretical analysis of \eqref{tck}. 
In general, the cost function $J(K)$ in problem  \eqref{tck} as 
well as the set of policies satisfying \eqref{eq:dom_ext} are non-convex.
In order to cope with this problem of non-convexity, we use 
the so-called PL condition named after Polyak and  Lojasiewicz, 
which is a relaxation of the notion of strong convexity. 
The PL condition is satisfied if there exists a universal 
constant $\mu> 0$ such that for any policy $K$ with finite $J(K)$,
we have   
\be\label{PLcond}
\Vert\nabla J(K)\Vert_F^2\geq \mu[J(K)- J(K^*)],
\ee
where $K^*$ denotes a global optimum for \eqref{tck}.
From Lemma 3 in \cite{Fazel2018Global}, Lemma 4 in \cite{malik2018derivative}, $R\succ 0$
and $\Sigma_{K,\gamma}\succeq \gamma D_{\omega}\succ 0$,
we get that \eqref{PLcond} is satisfied for \eqref{tck} with 
\be\label{PLcond_constant}
\mu=\frac{\gamma
\sigma_{\mathrm{min}}(\Sigma_{K,\gamma})^2\sigma_{\mathrm{min}}(R)}{(1-\gamma)
\|\Sigma_{K^*,\gamma}\|_2}>0.
\ee
We note that $\|\Sigma_{K^*,\gamma}\|_F<\infty$ holds due to 
the optimality of $K^*$.

By the PL condition \eqref{PLcond}, we know that stationary 
points of \eqref{tck}, i.e. $\nabla J(K)=0$, are global minima.
Since $\Sigma_{K,\gamma}\succeq \gamma D_{\widetilde\omega}\succ 0$, a
policy $K$ is a stationary point if and only if 
\be
(R+\gamma B^\top P_{K,\gamma}B)K+\gamma B^\top P_{K,\gamma}A=0.
\ee

We can verify  that $K^*$ in \eqref{kstar} is the global minimum of the function $J$. 
The optimization problem \eqref{tck} may have more than one global minimum and,
unfortunately, it is hard to derive the analytic form of all optimal 
policies in terms of $A$, $B$, $S$, $R$ and $\gamma$.

To simplify the analysis, we assume that $ \omega_t$ is Gaussian, that is $ \omega_t\sim N(0,D_{\omega})$. This assumption  also guarantees that $\widetilde\omega_t\sim N(0,D_{\widetilde\omega})$.

\section{Policy evaluation}\label{sect:PE}

Given a fixed stable policy $\pi_K$, it is desirable to know the expectation of the corresponding
cummulative cost for an initial state $x_0=x$, which in some sense evaluates ``how good'' it is 
to be in the state $x$ under policy $\pi_K$. This expectation is given by the so-called
value function $V_K(x)$, which is defined below. Moreover, the expectation of the 
cummulative cost for taking an action $u_0=u$ in some initial state $x_0=x$ under policy 
$\pi_K$ is given by state-action value $Q_K(x,u)$ which is also defined below, see also \cite{sutton2018reinforcement}.

The task of the policy evaluation is to get good approximations of the value function of a fixed stable policy $\pi_K$. In value-based methods, the policy evaluation is a very important and elementary step. In addition, the policy evaluation plays the role of the critic in the actor critic algorithm. The TD-learning method \cite{sutton1988learning} is a prevalent method for the policy evaluation. In this section, we discuss the usage of the TD-learning method in the LQR-setting.

First, we compute the value function $V_K$ and state-action value function $Q_K$ for a stable policy $\pi_K$. The value function, which gives the state value for the policy $\pi_K$ has the following explicit form: 
\be\label{valuef}
\begin{aligned}
&V_K(x):=\mathbb{E}_{p_K}\left[\sum_{t=0}^\infty\gamma^tc_t\Big|x_0=x\right]=x^\top P_{K,\gamma}x+\frac{\gamma}{1-\gamma}\Tr[P_{K,\gamma}D_{\widetilde\omega}]+\frac{\sigma^2\Tr[R]}{1-\gamma},
\end{aligned}
\ee
and the state-action value of the policy $\pi_K$ is given by
\be\label{sav}
\begin{aligned}
Q_{K}(x,u):=&\mathbb{E}_{p_K}\left[\sum_{t=0}^\infty\gamma^tc_t\Big|x_0=x,u_0=u\right]\\
=&\frac{\gamma}{1-\gamma}\Tr[\gamma P_{K,\gamma}D_{\widetilde\omega}+\sigma^2 R]+ 
\gamma\Tr[P_{K,\gamma}D_{\omega}]\\
&+[x^\top\;u^\top]
\left[
\begin{array}{cc}
S+\gamma A^\top P_{K,\gamma}A&\gamma A^\top P_{K,\gamma}B\\
\gamma B^\top P_{K,\gamma}A& R+\gamma B^\top P_{K,\gamma}B
\end{array}
\right]
\left[
\begin{array}{c}
x\\
u
\end{array}
\right],
\end{aligned}
\ee
where $p_K$ is the trajectory distribution of the system with policy $\pi_K$. 

As we will see below, the value function of $\pi_K$ can be written as a
function which is linear with respect to the feature function
\[
\phi(x)= \left[
\begin{array}{c}
1\\
\mathrm vec(xx^\top)
\end{array}
\right]
\]
for $x\in \mathbb{R}^{n}$, where $\mathrm vec$ denotes the vectorization operator stacking the columns of some matrix A on top of one another, i.e., 
\[\mathrm{vec}(A)=[a_{1,1},\cdots,a_{m,1},a_{1,2},\cdots,a_{m,2},\cdots,a_{1,n},\cdots,a_{m,n}]^\top.\]
Therefore, it is natural to propose a class of
linear approximation funtions with feature $\phi$ by
\[\widetilde V(x;\theta)=\phi(x)^\top\theta,\]
where $\theta=\left[
\begin{array}{c}
\theta_0\\
\theta_1
\end{array}
\right]$ is the parameter we seek to estimate.

To this end, we observe that $V_K(x)=\widetilde V(x,\theta^\star)$, 
where $\theta_1^\star = \mathrm vec(P_{K,\gamma})$ and 
$\theta_0^\star= \frac{\gamma}{1-\gamma}\Tr[P_{K,\gamma}D_{\widetilde\omega}]+\frac{\sigma^2\Tr[ R]}{1-\gamma}$.
For each stable policy $K$, our aim is to find a parameter $\theta$ such that
$\tilde V(x;\theta)$ approximates well $V_K(x)$ in expectation. This is 
carried out by minimizing a suitable loss function, for which we have to 
find a ``good'' distribution.
A reasonable choice for this distribution is the stationary distribution 
$N(0,D_{K})$ where
\[
D_{K}=\sum_{i=0}^\infty [(A+BK)^i]^\top D_{\widetilde\omega}(A+BK)^i
\]
for any stable $K$. It is easy to verify that if $x\sim N(0,D_{K})$, then $x'=(A+BK)x+\widetilde\omega\sim N(0,D_K)$ as well. This explain why we call it the stationary distribution. We also use the distribution $\mu_K$ to represent $N(0,D_K)$ in short.
\begin{remark}
For the sake of simplicity, we introduce the abbreviation
\begin{align*}
\mathbb{E}_{x,x'}[\cdot]:=\mathbb{E}_{x\sim\mu_K, \eta\sim N(0,\sigma^2I_d),\omega\sim N(0,D_{\omega})}[\cdot],
\end{align*}
since $x'=(A+BK)x+\eta+\omega$. When there is only the variable $x$ in the expection, we write $\mathbb E_{\mu_K}[\cdot]$ instead of $\mathbb{E}_{x,x'}[\cdot]$ to represent this expectation.  Otherwise, we use the notation $\mathbb E_{x,x'}[\cdot]$ for convenience.
\end{remark}
 Then we define the loss function
\[
L(\theta)=\frac12\mathbb{E}_{x\sim\mu_K}\left[(V_K(x)-\widetilde V(x;\theta))^2\right]
\]
such that the gradient of the loss function is given by
\[
\nabla_{\theta} L(\theta) = \mathbb{E}_{x\sim\mu_K}\left[\phi(x)(\widetilde V(x;\theta)-V_K(x))\right].
\]

However, in practice, since the real value function $V_K$ is unknown, people often use the bias estimation of  $c(x,u)+\gamma \widetilde V(x';\theta)$ to replace $V_K(x)$ where $x'=(A+BK)x+\widetilde \omega$ is the subsequent state of $x$. For convenience, let $\tilde\delta(x,u,x';\theta)=\widetilde V(x;\theta)-c(x,u)-\gamma \widetilde V(x';\theta)$, which is called TD error. 

We further note that $u=Kx+\eta$. Then we  obtain the semi-gradient
\be\label{evgrad}
\bar h(\theta)=\left[
\begin{array}{c}
\bar h_0(\theta)\\
\bar h_1(\theta)
\end{array}
\right]
=\mathbb{E}_{x,x'}\left[\phi(x)\tilde \delta(x,u,x';\theta)\right],
\ee 
and it is quite straightforward to get the stochastic semi-gradient
\be\label{vgrad}
h(\theta)=\left[
\begin{array}{c}
h_0(\theta)\\
h_1(\theta)
\end{array}
\right]= \left[\phi(x)\tilde \delta(x,u,x';\theta)\right].
\ee 
Starting with some initial parameter $\theta^{(0)}$ and using the gradient 
descent method with $\bar h(\theta)$ or $h(\theta)$ as update 
strategy, the TD learning method with linear approximation 
functions is described by \cref{tdlearn}.


\begin{algorithm}[H]
	\caption{TD learning}
\label{tdlearn}
	\begin{algorithmic}[1]
		\REQUIRE Stable policy $K$, parameters $\theta^{(0)}$, step size $\alpha$, the number of steps $N$.
\FOR {$s=1,2,\cdots,N-1$}
\IF{semi-gradient}
\STATE Compute the sample gradient $\bar h(\theta^{(s-1)})$ by \eqref{evgrad}.
\STATE Update the approximate value function $\theta^{(s)}=\theta^{(s-1)}-\alpha  \bar h(\theta^{(s-1)})$.
\ELSIF{stochastic semi-gradient}
\STATE $x\sim\mu_K$, $\widetilde\omega\sim N(0,D_{\widetilde\omega})$ and $x'=(A+BK)x+\tilde\omega$.
\STATE Compute the sample gradient $h(\theta^{(s-1)})$ by \eqref{vgrad}.
\STATE Update the approximate value function $\theta^{(s)}=\theta^{(s-1)}-\alpha   h(\theta^{(s-1)})$.
\ENDIF
\ENDFOR 
\STATE Averaging: $\tilde \theta =\frac1N \sum_{s=0}^{N-1}\theta^{(s)}$.
	\end{algorithmic}	
\end{algorithm}

The following lemma shows that $\theta^\star$ is a minimum of $\bar h$.
\begin{lemma}\label{lemma:hstar}
For any fixed stable policy $K$, we have
\be
\bar h(\theta^\star)=0.
\ee
\end{lemma}
\begin{proof}
For the first element in the semi-gradient, we get \[\bar h_0(\theta^\star)=\mathbb E_{x,x'}[\tilde \delta(x,u,x';\theta^\star)]=\mathbb E_{x,x'}[V_K(x)-c(x,u)-\gamma V_K(x')]=0,\]
by the  definition of $V_K$.
For the other elements we will consider them in matrix form
\[
\begin{aligned}
&\mathbb E_{x,x'}\left\{xx^\top [V_K(x)-c-\gamma V_K(x')]\right\}\\
=&\mathbb E_{\mu_K}\left\{xx^\top [x^\top P_{K,\gamma}x+\theta_0^\star-x^\top (S+K^\top RK)x-\sigma^2\Tr[ R]\right.\\
&\quad\quad\quad\quad\quad\left.-\gamma x^\top (A+BK)^\top P_{K,\gamma}(A+BK)x-\gamma\Tr[P_{K,\gamma}D_{\widetilde\omega}]-\gamma\theta_0^\star]\right\},
\end{aligned}
\]
where we have used that $\eta$ is independent of $x$.
By the definition of $P_{K,\gamma}$, we have $P_{K,\gamma}=S+K^\top RK+\gamma(A+BK)^\top P_{K,\gamma}(A+BK)$. For the constant term, it holds $(1-\gamma)\theta_0^*=\sigma^2\Tr[ R]+\gamma \Tr[P_{K,\gamma}D_{\widetilde\omega}]$. Thus, we obtain $\bar h(\theta^\star)=0$.
\end{proof}
\begin{assumption}\label{assb}
Assume that there exists $M_\theta>0$ such that the norms of  $\theta^\star$ and $\theta^{(s)}$ generated by \cref{tdlearn} are  bounded by $M_\theta$:
\[
\|\theta^\star\|_2\leq M_\theta, \quad  \|\theta^{(s)}\|_2\leq M_\theta.
\]
\end{assumption}
\begin{remark}
In some related works, people use projected gradient descent to guarantee the boundedness of $\theta^{(s)}$ which is not actually used in practice.
\end{remark}

\begin{lemma}
\label{lem:vtop}
There exists a positive number $\kappa$, such that for all $\theta$ and $\theta^*$
\be\label{mkappa}
\kappa\|\theta^*-\theta\|_2^2\leq\mathbb{E}_{\mu_K}\left[|\widetilde V(x,\theta^{*})-\widetilde V(x;\theta)|^{2}\right]\ee 
holds. In addition, $\kappa$ is continuous with respect to $K$.
\end{lemma}


\begin{proof}

Set $\theta^*-\theta= \left[
\begin{array}{c}
\theta_0^*-\theta_0\\
\mathrm vec(\Theta)
\end{array}
\right]$, where $\Theta$ is symmetric.

By Law of total variance,
\begin{align}\label{eq:variance}
\begin{aligned}
&\mathrm{Var}(x^{\top}\Theta x)\\
=&\mathbb{E}_{\mu_K}\left[x^{\top}\Theta xx^{\top}\Theta x\right]-\left[\mathbb{E}_{\mu_K}(x^{\top}\Theta x)\right]^2=2\Tr[D_K\Theta D_K\Theta]=2\|D_K^{1/2}\Theta D_K^{1/2}\|^2_F.
\end{aligned}
\end{align}
The second equality is derived by  Proposition A.1. in \cite{tu2017least}.

Thus, we can define \[\lambda_K=\inf_{\|\Theta\|_F=1}\|D_K^{1/2}\Theta D_K^{1/2}\|_F.\]
Since norms of matrices are equivalent and $\mathrm{Var}(x^{\top}\Theta x)>0$ for $\Theta\neq 0$, we obtain from \eqref{eq:variance} that $\lambda_K$ is positive and continuous
with respect to $K$. 
Besides, $\left|\mathbb{E}_{\mu_K}[x^{\top}\Theta x]\right| = |\Tr(D_{K}\Theta)|\leq \|D_{K}\|_F\|\Theta\|_F$. Using this inequality in connection with \eqref{eq:variance} we obtain

\[
\begin{aligned}
&\mathbb{E}_{\mu_K}\left[|\widetilde V(x,\theta^*)-\widetilde V(x;\theta)|^{2}\right]\\
= &\mathbb{E}_{\mu_K}\left[(x^\top\Theta x+\theta_0^*-\theta_0)^2\right]\\
=&\mathrm{Var}(x^{\top}\Theta x)+\left(\theta_0^*-\theta_0+\mathbb{E}_{\mu_K}\left[x^\top\Theta x\right]\right)^2\\
\geq& 2\lambda_K^2 \|\Theta\|_F^2+\left(\theta_0-\theta_0^*-\mathbb{E}_{\mu_K}\left[x^\top\Theta x\right]\right)^2\\
\geq&{\lambda^2_K}\|\Theta\|_F^2+\frac{\lambda^2_K}{\|D_{K}\|_F^2}\left(\mathbb{E}_{\mu_K}\left[x^{\top}\Theta x\right]\right)^2+\left(\theta_0-\theta_0^*-\mathbb{E}_{\mu_K}\left[x^\top\Theta x\right]\right)^2\\
\geq& \lambda_K^2\|\Theta\|_F^2+\left(\frac{1}{1+\frac{\|D_{K}\|^2_F}{\lambda^2_K}}\right)\left(-\mathbb{E}_{\mu_K}\left[x^{\top}\Theta x\right]+\theta_0-\theta_0^*+\mathbb{E}_{\mu_K}\left[x^{\top}\Theta x\right]\right)^2\\
\geq& {\lambda_K^2}\|\Theta\|_F^2+\left(\frac{\lambda_K^2}{\lambda_K^2+\|D_K\|_F^2}\right)(\theta_0^*-\theta_0)^2.
\end{aligned}
\]
where the third inequality follows by the Cauchy-Schwarz inequality.

Thus, we conclude the existence of such a positive number $\kappa$ which 
is continuous with respect to $K$.
\end{proof}

Under \cref{assb}, the constant $\kappa$ has a uniform lower bound in a compact subset of $\mathbf{Dom}$.
Then we can use some basic tool of the analysis of the gradient descent method to prove the convergence of \cref{tdlearn}.
\begin{theorem}\label{thm:semi-gradient_update}
Suppose that \cref{assb} holds. \cref{tdlearn} is run with the semi-gradient update using $\alpha=\frac{1-\gamma}{2(1+(n+2)\|D_K\|_F^2)}$ to generate $\widetilde V(x; \tilde \theta)$. Then, the following inequality holds for all $x\in\mathbb{R}^n$:
\be\label{bound4-1}
\mathbb{E}_{\mu_K} [(\widetilde V(x;\tilde \theta) - \widetilde V(x;\theta^\star))^2]\leq \frac{8(1+(n+2)\|D_K\|_F^2)M_\theta^2}{(1-\gamma)^2N}.
\ee

\end{theorem}
\begin{proof}
By the definition of $h(\theta)$ and the fact that $\|\phi(x)\|_2^2=(\|x\|_2^4+1)$, we know
\be\label{sigmahb}
\begin{aligned}
\|h(\theta)\|^2_2=&(c+\gamma \phi(x')^\top\theta-\phi(x)^\top\theta)^2\|\phi(x)\|_2^2\\
\leq &2c^2 \|\phi(x)\|_2^2+2M_\theta^2\|\gamma\phi(x')-\phi(x)\|^2_2\|\phi(x)\|_2^2\\
=&2(x^\top(S + K^\top RK)x + 2\eta^\top RKx +\eta^\top R\eta)^2(\|x\|_2^4+1)\\
&+2M_\theta^2\left[(1-\gamma)^2+2\|x'\|_2^4+2\|x\|_2^4\right](\|x\|_2^4+1).
\end{aligned}
\ee
Thus, $\|h(\theta)\|^2_2$ is an $8$th order polynomial with respect to $x$, $\omega$ and $\eta$.
Then we can obtain the bound $\sigma_h^2=O\left((M^2_\theta +\|K\|_2^4)\mathbb E_{\mu_K}\|x\|_2^8+\mathbb E_{\mu_K}\|x\|^4_2\sigma^4\right)$.
\end{proof}

We next show that \cref{tdlearn} also converges if the stochastic semi-gradients are used.  
\begin{theorem}\label{thm:state_value_estimation_stochastic}

Suppose that \cref{assb} holds. Let \cref{tdlearn} be run with the stochastic semi-gradient update using $\alpha =\min\{\frac{1 - \gamma}{4(1+(n+2)\|D_K\|_F^2)},1/\sqrt{N}\}$ to generate $\widetilde V(x; \tilde \theta)$. Then,  for all $x\in\mathbb{R}^n$ it holds 
\be
\mathbb{E}_{\mu_K} [(\widetilde V(x;\widetilde \theta) - \widetilde V(x;\theta^\star))^2]\leq  \frac{M_\theta^2+2\sigma_h^2}{2(1-\gamma)\sqrt{N}-4(1+(n+2)\|D_K\|_F^2)}.
\ee
\end{theorem}
\begin{proof}
The key step to prove \eqref{bound4-1} is based on an important gradient descent inequality.
In order to establish this inequality, we have to estimate two terms at first.
For the descent term $(\bar h(\theta^{(s)})-\bar h^\star)^\top(\theta^{(s)}-\theta^\star)$, we get the following lower bound:
\be\label{bound4-2}
\begin{aligned}
&(\bar h(\theta^{(s)})-\bar h^\star)^\top(\theta^{(s)}-\theta^\star)
\\=&\mathbb{E}_{x,x'}\{[\tilde\delta(x,u,x';\theta^{(s)})-\tilde\delta(x,u,x';\theta^\star)]\phi(x)^\top(\theta^{(s)}-\theta^\star)\}\\
=&\mathbb{E}_{\mu_K}[\widetilde V(x;\theta^{(s)})-\widetilde V(x;\theta^\star)]^2-\gamma\mathbb{E}_{x,x'}(\widetilde V(x;\theta^{(s)})-\widetilde V(x;\theta^\star))(\widetilde V(x';\theta^{(s)})-\widetilde V(x';\theta^\star))\\
\ge& (1-\gamma)\mathbb{E}_{\mu_K}[\widetilde V(x;\theta^{(s)})-\widetilde V(x;\theta^\star)]^2,
\end{aligned}
\ee
where last inequality follows by Cauchy-Schwarz inequality. Then we split the norm of  $\bar h(\theta^{(s)})-\bar h^\star$ into two parts:
\be\label{bound4-11}
\begin{aligned}
\|\bar h(\theta^{(s)})-\bar h^\star\|^2_2=&\Big\|\mathbb E_{\mu_K}\big\{\phi(x)[\widetilde V(x;\theta^{(s)})-\widetilde  V(x;\theta^\star)+\gamma\widetilde V(x';\theta^\star)-\gamma\widetilde  V(x';\theta^{(s)})]\big\}\Big\|_2^2\\
\leq& \Big\|\mathbb E_{x,x'}\big\{\phi(x)[\widetilde V(x;\theta^{(s)})-\widetilde  V(x;\theta^\star)]\big\}\Big\|_2^2\\
&+\Big\|\mathbb E_{x,x'}\big\{\gamma\phi(x)[\widetilde V(x';\theta^{(s)})-\widetilde  V(x';\theta^\star)]\big\}\Big\|_2^2=I_1+I_2.
&
\end{aligned}
\ee
Then we get the upper bound of $I_1$ and $I_2$ by the Cauchy-Schwarz inequality:

\be
\begin{aligned}
I_1&=\sum_{j}\Big|\mathbb E_{\mu_K}\big\{\phi_j(x)[\widetilde V(x;\theta^{(s)})-\widetilde  V(x;\theta^\star)]\big\}\Big|^2\\
&\leq \sum_{j}\mathbb E_{\mu_K}[\phi_j(x)]^2\mathbb E_{\mu_K}[\widetilde V(x;\theta^{(s)})-\widetilde  V(x;\theta^\star)]^2\\
&=(1+\mathbb E_{\mu_K}\|xx^\top\|_F^2)\mathbb E_{\mu_K}[\widetilde V(x;\theta^{(s)})-\widetilde  V(x;\theta^\star)]^2\\
&=[1+2\|D_K\|_F^2+(\Tr[D_K])^2]\mathbb E_{\mu_K}[\widetilde V(x;\theta^{(s)})-\widetilde  V(x;\theta^\star)]^2
\end{aligned}
\ee
In particular, the last equality holds by Proposition A.1. in \cite{tu2017least}.
Analogously, $I_2$ has the same upper bound such that
\be\label{bound4-3}
\|\bar h(\theta^{(s)})-\bar h^\star\|^2_2\leq 2[1+(n+2)\|D_K\|_F^2][\widetilde V(x;\theta^{(s)})-\widetilde  V(x;\theta^\star)]^2,
\ee
where we used $(\Tr[D_K])^2\leq n \|D_K\|_F^2$.

By \eqref{bound4-2}, \eqref{bound4-3} and \cref{lemma:hstar}, we obtain 
\be
\label{egdd}
\begin{aligned}
&\|\theta^{(s+1)}-\theta^\star\|_2^2 = \|\theta^{(s)}-\alpha\bar h(\theta^{(s)})-\theta^\star\|_2^2\\
= &\|\theta^{(s)}-\theta^\star\|_2^2 -2\alpha(\bar h(\theta^{(s)})-\bar h^\star)^\top(\theta^{(s)}-\theta^\star)+\alpha^2\|\bar h(\theta^{(s)})-\bar h^\star\|^2_2\\
\leq & \|\theta^{(s)}-\theta^\star\|_2^2-(2\alpha(1-\gamma)-2(1+(n+2)\|D_K\|_F^2)\alpha^2)\mathbb{E}_{\mu_K}[\widetilde V(x;\theta^{(s)})-\widetilde V(x;\theta^\star)]^2.
\end{aligned}
\ee
The value of the right hand of \eqref{egdd} is minimal for $\alpha=\frac{1-\gamma}{2(1+(n+2)\|D_K\|_F^2)}$. Rearranging \eqref{egdd} and summing the inequalities from $t=0$ to $N-1$ yields:
\be
\begin{aligned}
&\mathbb{E}_{\mu_K}[\widetilde V(x;\tilde\theta)-\widetilde V(x;\theta^\star)]^2\\
\leq& \frac1N\sum_{s=0}^{N-1}\mathbb{E}_{\mu_K}[\widetilde V(x;\theta^{(s)})-\widetilde V(x;\theta^\star)]^2\\
\leq &\frac{1}N\sum_{t=0}^{N-1}\frac{2(1+(n+2)\|D_K\|_F^2)}{(1-\gamma)^2}(\|\theta^{(s)}-\theta^\star\|^2_2-\|\theta^{(s+1)}-\theta^\star\|^2_2)\\
\leq& \frac{8(1+(n+2)\|D_K\|_F^2)M_\theta^2}{(1-\gamma)^2N}.
\end{aligned}
\ee
This completes the proof.
\end{proof}

Furthermore, $\mathbb E_{\mu_K}\|x\|_2^{2k}$ has a uniform upper bound with respect to $K$ on a compact subset of $\mathbf{Dom}$, since $\mathbb E_{\mu_K}\|x\|_2^{2k}=O(\frac{\Gamma_K^{2k}}{1-\rho^{2k}}\mathbb E\|\widetilde\omega\|^{2k}_2])$ roughly.
Then we can derive convergence of $\theta^{(s)}$ under the $2-$norm by \cref{lem:vtop}.

\section{Policy Iteration}\label{sect:PI}
The policy iteration (PI) \cite{lagoudakis2003least} is a very fundamental value-based method in Markov decision process (MDP) problems with finite actions. This method can also be applied to LQRs, since the state-action value function $Q$ is quadratic and it is easy to find the best action such that the value function $Q$ is minimal by solving linear equations.

Given a policy $\pi_K$ and the corresponding state-action function $Q_K$ of the form \eqref{sav}, the policy $\pi_K$ can be improved by selecting the action $u^*$ for a fixed state $x$ such that the value of $Q_K$ is minimal:
\[
\begin{aligned}
u^* =& \mathrm{arg }\min_{u}Q_K(x,u)\\
=&\mathrm{arg }\min_{u} [x^\top\;u^\top]
\left[
\begin{array}{cc}
S+\gamma A^\top P_{K,\gamma}A&\gamma A^\top P_{K,\gamma}B\\
\gamma B^\top P_{K,\gamma}A& R+\gamma B^\top P_{K,\gamma}B
\end{array}
\right]
\left[
\begin{array}{c}
x\\
u
\end{array}
\right]\\
=&-(R+\gamma B^\top P_{K,\gamma}B)^{-1}(\gamma B^\top P_{K,\gamma}A)x.
\end{aligned}
\]
Thus, we obtain an improved policy $\pi_{K'}$ with
\be\label{piex}
K'=-(R+\gamma B^\top P_{K,\gamma}B)^{-1}(\gamma B^\top P_{K,\gamma}A).
\ee
Observing the gradient \eqref{exgradient2}, there is another form of \eqref{piex}:
\be
\label{pi}
K' = K-\frac12(R+\gamma B^\top P_{K,\gamma}B)^{-1}\nabla J\Sigma_{K,\gamma}^{-1}.
\ee
This form corresponds to the Gauss-Newton method in \cite{Fazel2018Global} with stepsize $\frac{1}{2}$. Hence, by the discussion in \cite{Fazel2018Global}, we obtain the convergence of the policy iteration with the exact state-action value function $Q_K$ and in addition the fact that value based methods are faster than the policy gradient method. Moreover, from Lemma~8 in \cite{Fazel2018Global}, we obtain that $J(K')<J(K)$.
This implies that if $K$ is stable, then also $K'$ is stable provided that $
\gamma$ is sufficiently close to $1$, i.e., if $\gamma$ satisfies 
\eqref{eq:lemma_restriction_to_dom2} with $\nu=1$, see \cref{lem:restriction_to_dom}.

However, we do not know the state action value function $Q$ exactly since we do not know the system exactly. Due to the results of the former discussion of policy evaluation, we can obtain an approximation of $Q$. This approximation is used in a policy iteration scheme, for which we want to analyze the convergence. The approximation of the state action value function has the following form: 
\[\widetilde Q(x,u;\Theta)=[x^\top\;u^\top]
\left[
\begin{array}{cc}
\Theta_{11}&\Theta_{12}\\
\Theta_{21}&\Theta_{22}
\end{array}
\right]
\left[
\begin{array}{c}
x\\
u
\end{array}
\right]+\Theta_0.\]
For any stable policy $\pi_K$, we denote the parameters of $Q_K$ by $\Theta^*_K$.
If $\|\Theta-\Theta^*_{K}\|_F\leq \epsilon_0$, we call $\widetilde Q(x,u;\Theta)$ an $\epsilon_0-$approximation of the state-action value function. The policy can also be improved by using the following approximation for the state action value:
\be\label{piap}
K'=-\Theta^{-1}_{22}\Theta_{21}.
\ee
Thus, we obtain the policy iteration algorithm.
\begin{algorithm}[H]
	\caption{The policy iteration method}
	\label{PIin}
	\begin{algorithmic}[1]
		\REQUIRE Stable policy $K^{(0)}$, the number of steps $T$.
		
		\FOR {$s=1,2,\cdots,T$}
		\STATE Evaluate an approximate state-action value function $\widetilde Q(\cdot,\Theta^{(s-1)})$ of policy $K^{(s-1)}$ by TD learning with  the stochastic gradient.
		
		\STATE Policy improvement: 
		\[K^{(s)}=-(\Theta_{22}^{(s-1)})^{-1}\Theta_{21}^{(s-1)}.
		\]
		\ENDFOR 
		
	\end{algorithmic}	
\end{algorithm}
As we will prove below, the policy iteration with approximate state-action value functions
also converges in the LQR-setting if the error between the approximation value and the true value is small enough. For any stable initial policy $K^{(0)}$ and any $\epsilon>0$, we assume that the error tolerance $\epsilon_0$ between the approximation value $\Theta$ and the true value $\Theta^{\star}$ satisfies 
\be\label{cond5-1}
\epsilon_0\leq \min\left\{\frac{1}{2}\|R^{-1}\|_2,\frac{C_1\sqrt{\epsilon}}{\sqrt{\|\Sigma_{K^*,\gamma}\|_2}J(K^{*})}\right\}
\ee 
where $C_1$ only depends on $\sigma_{\mathrm{min}}(D_{\widetilde \omega})$, $\sigma_{\mathrm{min}}(S)$, $\|R^{-1}\|_2$, $\|A\|_2$, $\|B\|_2$ and $(1-\gamma)J(K^{(0)})$.
\begin{theorem}\label{thm:policy_iteration_convergence}
	For any stable initial policy $K^{(0)}$ and any $\epsilon>0$, suppose that an $\epsilon_0-$approximation of the state action value function is known, where $\epsilon_0$ satisfies \eqref{cond5-1}. If $\gamma$ is sufficiently close to $1$, then for $T\ge\left[\frac{2\|\Sigma_{K^*,\gamma}\|_2}{\gamma^2\sigma_{min}(D_{\widetilde \omega})}-1\right]\log\frac{J(K_{0})-J(K^*)}{\epsilon}$, \cref{PIin} yields stable policies $K^{(1)},\ldots, K^{(T)}$ which are elements of \textbf{Dom}$_{K^{(0)},2}$ and satisfy
	\be\label{eq:thm_policy_iteration_convergence}
	J(K^{(T)})-J(K^*)\leq \epsilon.
	\ee
\end{theorem}
\begin{proof}
We first introduce some notation which is used in this proof. Let $\widetilde Q(\cdot,\Theta)$ be an $\epsilon_0-$approximation of the state action value function $Q_K(\cdot)$ and 
    denote by $\Theta^{\star}$ the true value, i.e., $Q_K(\cdot)=\widetilde Q(\cdot,\Theta^{\star})$. Then the improved policy is given by $K'=-\Theta_{22}^{-1}\Theta_{21}$.
	Next, we obtain from Lemma~6 in \cite{Fazel2018Global} that
	\be\label{diffK}
	\begin{aligned}
		&J(K')-J(K)\\
		=&\Tr[2\Sigma_{K',\gamma}(K'-K)^\top E_K+\Sigma_{K',\gamma}(K'-K)^\top (R+\gamma B^\top P_{K,\gamma}B)(K'-K)],
	\end{aligned}
	\ee
	where $E_K=(R+\gamma B^\top P_{K,\gamma}B)K+\gamma B^\top P_{K,\gamma}A$. Furthermore, we note that $R+\gamma B^\top P_{K,\gamma}B=\Theta^*_{22}$ and $\gamma B^\top P_{K,\gamma}A=\Theta^*_{21}$ holds by \eqref{piex}. Let $\Delta_1 := \Theta_{21}-\Theta_{21}^*$ and $\Delta_2:=\Theta_{22}-\Theta_{22}^*$. 
	Then we compute the difference between $K'$ and the policy $-(\Theta^{*}_{22})^{-1}\Theta^{*}_{21}$ generated by the true policy iteration: 
	\be\label{diffkk2}
	\begin{aligned}
		K'+(\Theta^{*}_{22})^{-1}\Theta^{*}_{21}=&-(\Theta^{*}_{22})^{-1}(I+\Delta_2(\Theta^{*}_{22})^{-1})^{-1}(\Theta^{*}_{21}+\Delta_1)+(\Theta^{*}_{22})^{-1}\Theta^{*}_{21}\\
		=&(\Theta^{*}_{22})^{-1}(I+\Delta_2(\Theta^{*}_{22})^{-1})^{-1}(\Delta_2(\Theta^{*}_{22})^{-1}\Theta^{*}_{21}-\Delta_1)\\
		=&(\Theta^{*}_{22})^{-1}\Delta_3
	\end{aligned}
	\ee
	with $\Delta_3=(I+\Delta_2(\Theta^{*}_{22})^{-1})^{-1}(\Delta_2(\Theta^{*}_{22})^{-1}\Theta^{*}_{21}-\Delta_1)$. 
	Next, we note that \eqref{pi} implies 
	\be\label{eq:diff_K}
	K+(\Theta^{*}_{22})^{-1}\Theta^{*}_{21}=(\Theta_{22}^{\star})^{-1}E_K
	\ee
	which together with \eqref{diffkk2} yields
	\be\label{diffkk}
	K'-K=(\Theta^{*}_{22})^{-1}(\Delta_3-E_K)\ee Since $\|\Theta-\Theta_K^*\|_F\leq \epsilon_0$, it holds $\|\Delta_1\|_F\leq \epsilon_0$, $\|\Delta_2\|_F\leq \epsilon_0$ and $\|\Delta_3\|_F$ has the upper bound
	\be\label{delta3_estimation}
	\|\Delta_3\|_F\leq \frac{\epsilon_0}{1-\epsilon_0 \|R^{-1}\|_2}(\gamma\|R^{-1}\|_2\|B\|_2\|A\|_2\|P_{K,\gamma}\|_2+1)\leq C_0\epsilon_0(1-\gamma) J(K)+2\epsilon_0,
	\ee
	where the second inequality follows by \eqref{eq:J_est_1} with 
	\begin{align*}
	C_0=\frac{2\|R^{-1}\|_2\|B\|_2\|A\|_2}{\sigma_{\mathrm{min}}(D_{\widetilde \omega})}.
	\end{align*}
	Using \eqref{diffK} and \eqref{diffkk}, the differences of the cumulative cost between the original policy $K$ and the improved policy $K'$ can be represented as: 
	\be\label{diffup}
	\begin{aligned}
		J(K')-J(K)=\Tr[-\Sigma_{K',\gamma}E_K^\top (\Theta_{22}^*)^{-1}E_K+\Sigma_{K',\gamma}\Delta_3^\top(\Theta^{*}_{22})^{-1}\Delta_3].
	\end{aligned}
	\ee
	For the first term in \eqref{diffup}, we have
	\be\label{bound5-1}
	\begin{aligned}
		\Tr[\Sigma_{K',\gamma}E_K^\top (\Theta_{22}^*)^{-1}E_K]\ge& \sigma_{\mathrm{min}}(\Sigma_{K',\gamma})\Tr[E_K^\top (\Theta_{22}^*)^{-1}E_K]\\
		\geq& \frac{\gamma^2\sigma_{\mathrm{min}}(D_{\widetilde \omega})}{\|\Sigma_{K^*,\gamma}\|_2}(J(K)-J(K^*)),
	\end{aligned}
	\ee
	where the second inequality is derived by the fact $\Sigma_{K',\gamma}\succeq \gamma^2 D_{\widetilde \omega} $ and Lemma~11 in the supplementary material of \cite{Fazel2018Global}. More precisely, one can check that this lemma is also valid for the setting in this paper.
	For the second term in \eqref{diffup}, using \eqref{eq:J_est_2} and \eqref{delta3_estimation} we get the upper bound
	\be\label{bound5-2}
	\begin{aligned}
		&\Tr[\Sigma_{K',\gamma}\Delta_3^\top (\Theta^*_{22})^{-1}\Delta_3]\\
		\le& \|\Sigma_{K',\gamma}\|_F\Tr[\Delta_3^\top(\Theta^{*}_{22})^{-1}\Delta_3]\le \|\Sigma_{K',\gamma}\|_F\|R^{-1}\|_2\Tr[\Delta_3^\top\Delta_3]\\
		\le&\|R^{-1}\|_2\frac{
			[C_0(1-\gamma)J(K)+2]^2\epsilon_0^2}{\sigma_{\mathrm{min}}(S)}J(K')
	\end{aligned}
	\ee
	By \eqref{bound5-1} and \eqref{bound5-2}, it is direct to obtain the following inequality:
	\be\label{eq:policy_iteration_descent0}
	J(K')-J(K^*)\leq \frac{1-\alpha}{1-\beta}(	J(K)-J(K^*))+\frac{\beta}{1-\beta}J(K^*),
	\ee
	where $\alpha=\frac{\gamma^2\sigma_{\mathrm{min}}(D_{\widetilde \omega})}{\|\Sigma_{K^*,\gamma}\|_2}<1$ and $\beta = \|R^{-1}\|_2\frac{[C_0(1-\gamma) J(K)+2]^2\epsilon_0^2}{\sigma_{\mathrm{min}}(S)}$. 
	
	We can start with $K^{(0)}$ and set $\beta=\|R^{-1}\|_2\frac{[2C_0(1-\gamma) J(K^{(0)})+2]^2\epsilon_0^2}{\sigma_{\mathrm{min}}(S)}$. Then let $\epsilon_0$ be small enough such that $\beta\leq\alpha/2$.
	Thus, the bound of $\epsilon_0$ is $O(\frac{1}{2C_0(1-\gamma)J(K^{(0)})+2})$. Next, we show that \be\label{bound:unf-bound1} J(K^{(t)})\le 2J(K^{(0)}),\ee which implies that the inequality \eqref{eq:policy_iteration_descent0} holds with $K = K^{(t)}$ and $K'=K^{(t+1)}$ for all iterates.
	
	We use induction to prove this uniform bound \eqref{bound:unf-bound1}. When $t=0$, this inequality \eqref{bound:unf-bound1} holds obviously.
	Then we assume that \eqref{bound:unf-bound1} holds with $K^{(t)}$.
	Using $\beta\leq\alpha/2$ in connection with \eqref{eq:policy_iteration_descent0}, we observe that 
	if $J(K^{(t)})\geq 2J(K^*)$, then 
	\[J(K^{(t+1)})\leq \frac{1-\alpha}{1-\beta}J(K^{(t)})+(1-\frac{1-\alpha-\beta}{1-\beta})J(K^*)\le \frac{1-\alpha/2}{1-\beta}J(K^{t})\leq 2J(K^{(0)})\]
	by the inequality \eqref{eq:policy_iteration_descent0}.
	Otherwise it holds $J(K^{(t+1)})\leq \frac{2-\alpha}{1-\beta}J(K^*)<2J(K^{(0)})$.
	Thus, the bound \eqref{bound:unf-bound1} holds for all $t$.
	
	Hence, we have following inequality:
	\be\label{eq:policy_iteration_descent}
	J(K^{(t)})-J(K^*)\leq \left( \frac{1-\alpha}{1-\beta}\right)^t(J(K^{(0)})-J(K^*))+\frac{\beta}{\alpha-\beta}J(K^*).
	\ee
	Furthermore, we also require that $\frac{\beta}{\alpha-\beta}J(K^*)\leq \frac{2\beta}{\alpha}J(K^*)\leq \frac{\epsilon}2$, which is equivalent to $\beta\leq \frac{\alpha\epsilon}{4J(K^*)}$. Thus, the upper bound of $\epsilon_0$ should be $ C_1\frac{\sqrt{\epsilon}}{\sqrt{\|\Sigma_{K^*,\gamma}\|_2}J(K^*)}$, where 
	\[C_1=O\left(\frac{1}{[C_0(1-\gamma)J(K^{(0)})+2]}\right).\]
	Then we can verify that $\left( \frac{1-\alpha}{1-\beta}\right)^T(J(K^{(0)})-J(K^*))\leq \frac{\epsilon}2$ such that \eqref{eq:thm_policy_iteration_convergence} is proved. Finally, we  assume w.l.o.g. that $J(K^{(0)})-J(K^*)>\epsilon$. We can guarantee that $\frac{\beta}{\alpha-\beta}\leq \frac{\epsilon}{2J(K^*)}\le \frac{J(K^{(0)})-J(K^*)}{2J(K^*)}$. Inserting this in \eqref{eq:policy_iteration_descent}, we obtain
	\begin{align*}
	J(K^{(t)})-J(K^*)&\leq \left[\left( \frac{1-\alpha}{1-\beta}\right)^t+\frac12\right](J(K^{(0)})-J(K^*))\\
	&\leq 2(J(K^{(0)})-J(K^*))
	\end{align*}
	Hence,  $K^{(t)}\in $       \textbf{Dom}$_{K^{(0)},2}\subset$ \textbf{Dom}
    holds by \cref{lem:restriction_to_dom} for all iterates, if $\gamma$ is sufficiently close to $1$.
\end{proof}

Using the definition of $\alpha$, we observe that the approximation parameter $\epsilon_0$ has an upper bound $O\left({(1-\gamma)^{\frac32}\epsilon^{\frac12}}\right)$. Thus, for each TD-learning of a fix $K^{(t)}$, it needs $O\left(\frac1{(1-\gamma)^5\epsilon}\right)$ samplings by \cref{lem:vtop} and \cref{thm:semi-gradient_update}. However if we use  stochastic semi-gradient descent, the sample complexity for each policy evaluation of this algorithm becomes   $O\left(\frac1{(1-\gamma)^8\epsilon^2}\right)$ by \cref{thm:state_value_estimation_stochastic}.	
\section{The Policy Gradient method}\label{sect:PG}
In RL, the policy gradient method \cite{williams1992simple,Sutton1999Policy} is widely used. In this section, we apply the policy gradient method to the problem  \eqref{tck} and analyze the convergence of this method. 

In order to compute the policy gradient, we have to know the score function $\nabla_K\log\pi_K$ of the policy $\pi_K$, which is given by
\be
\begin{aligned}
\quad\nabla_K\log\pi_K(u|x)=\frac{1}{\sigma^2}(u-Kx)x^\top.
\end{aligned}
\ee
By the policy gradient theorem in \cite{Sutton1999Policy}, we obtain the policy gradient
\be\label{gradient}
\begin{aligned}
G(K)&=\frac{1}{\sigma^2}\mathbb{E}_{p_K}\left[\sum_{t=0}^{\infty}(u_t-Kx_t)x_t^\top \sum_{k=t}^\infty\gamma^kc_k\right]\\
&=\frac{1}{\sigma^2}\mathbb{E}_{p_K}\left[\sum_{t=0}^{\infty}\gamma^t(u_t-Kx_t)x_t^\top Q_{K}(x_t,u_t)\right].
\end{aligned}
\ee
The policy gradient \eqref{gradient} is equivalent to $\nabla J(K)$ which is shown in Lemma 5.2.
For the representation of the gradient in \eqref{gradient} it is straightforward to design
an estimation.
After achieving some triples $\{x_t,u_t,c_t\}_{t=0}^L$ generated by the policy $\pi_K$ in problem $\eqref{tck}$, we can compute the sample gradient:
\be\label{sg}
\hat G^{(L)}(K)=\frac{1}{\sigma^2}\sum_{t=0}^L\left[(u_t-Kx_t)x_t^\top \sum_{k=t}^L\gamma^kc_k\right].
\ee
We apply the stochastic gradient descent method to the problem \eqref{tck}, which is summarized in \cref{samplegrad}.

\begin{algorithm}[H]
	\caption{The Policy Gradient Method}
\label{samplegrad}
	\begin{algorithmic}[1]
		\REQUIRE Initial policy $K_{0}$, roll out length $L$, step size $\alpha$.

\FOR {$s=0,1,\cdots,T-1$}
            \STATE Simulate $u_t=K^{(s)}x_t+\eta_t$ for $L$ steps starting from $x_0=0$ and obtain $D^{(s)}=\{x^{(s)}_t,u^{(s)}_t,c^{(s)}_t\}_{t=0}^L$.
\STATE Compute $\hat G^{(L)}(K^{(s)})$ by \eqref{sg}.
\STATE Update policy $K^{(s+1)}=K^{(s)}-\alpha  \hat G^{(L)}(K^{(s)})$. 
\ENDFOR 
	\end{algorithmic}	
\end{algorithm}

Since the  estimator $\hat G^{(L)}(K)$ is biased, we have to control the bias by increasing the length $L$. The next lemma shows that  how the parameter $L$ influences the bias and the variance of the estimator $\hat G^{(L)}(K)$.

\begin{lemma}\label{bounded}
Let $K$ be a stable policy. Then it holds 
	\begin{eqnarray}
	\left\Vert \nabla J(K)-\mathbb{E}[\hat G^{(L)}(K)]\right\Vert_2&\leq& C_2\frac{\Gamma_K^2}{1-\rho^2}\gamma^{L+1}(\frac{\Gamma_K^2}{1-\rho^2}+\frac{1}{1-\gamma}),\label{lemmaeq5}\\
	\mathbb{E}\left\Vert \hat G^{(L)}(K)\right\Vert^2_F&\leq& C_3\frac{L^3}{(1-\gamma^2)},\label{lemmaeq6}
	\end{eqnarray} 
	where $\rho\in (\rho(A+BK),1)$ and the constants $C_2$ and $C_3$ depend on $\|A\|_2$, $\|B\|_2$, $\|S\|_2$, $\|R\|_2$, $\|K\|_2$, $\Gamma_K$, $\frac1{1-\rho^2}$, $\widetilde\omega$, $n$, $d$ and $(1-\gamma)J(K)$.
\end{lemma}
\begin{proof}
	Define the event fields $\mathcal F_t$ generated by $(x_0,x_1,\cdots,x_t)$ and the operator $\mathcal T(X)=(A+BK)^\top X(A+BK)$. We observe that $\eta_t$ and $\omega_t$ are independent from $\mathcal F_t$. By the definition of the value function in \eqref{valuef}, it is straightforward to obtain the conditional expection of the   cumulative cost:
	\[\mathbb{E}\left[\sum_{k=t}^\infty \gamma^{k}c_k\Big|\mathcal F_{t}\right]=\gamma^{t}V_K(x_{t})=\gamma^{t}x_{t}^\top P_{K,\gamma}x_{t}+\frac{\gamma^{t+1}}{1-\gamma}\Tr[P_{K,\gamma}D_{\widetilde\omega}]
	+\frac{\gamma^{t}\sigma^2\Tr[R]}{1-\gamma}.\] 
	
	First, we claim that $G(K)$ is equivalent to $\nabla J(K)$. To this end, we verify the following identity:
	\be
	\label{boundeqs1}
	\begin{aligned}
		&\mathbb{E}\left[(u_t-Kx_t)x_t^\top \sum_{k=t}^\infty\gamma^kc_k\Big|\mathcal F_t\right]\\
		=&\mathbb{E}\left[\eta_t x_t^\top \sum_{k=t+1}^\infty\gamma^kc_k\Big|\mathcal F_t\right]+\mathbb{E}\left[\eta_tx_t^\top \gamma^tc_t\Big|\mathcal F_t\right]\\
		=&\mathbb{E}\left[\mathbb{E}[\eta_tx_t^\top \sum_{k=t+1}^\infty \gamma^kc_k|\mathcal F_{t+1}]\Big|\mathcal F_{t}\right]\\
		&+\gamma^t\mathbb{E}\left[\eta_t (x_t^\top 	Sx_t+(Kx_t+\eta_t)^\top R(Kx_t+\eta_t))\Big|\mathcal F_t\right]x_t^\top\\
		=&\mathbb{E}\left[\eta_t\gamma^{t+1}V_K(x_{t+1})\Big|\mathcal F_{t}\right]x_t+2\sigma^2\gamma^tRKx_tx^\top_t\\
		=&\gamma^{t+1}\mathbb{E}\left[\eta_tx_{t+1}^\top P_{K,\gamma} x_{t+1}\Big|\mathcal F_{t}\right]x_t+2\sigma^2\gamma^tRKx_tx^\top_t\\
		=&2\sigma^2\gamma^t[\gamma B^\top P_{K,\gamma}(A+BK)+RK]x_tx_t^\top.
	\end{aligned}\ee
	The third equality is valid since $\mathbb{E}[\eta_t]=0$ and 
	\be\label{eq:symmetrie}
	\mathbb{E}\left[\eta_t\eta_t^{\top}B^{\top}P_{K,\gamma}B\eta_t\Big|\mathcal F_{t}\right]=0,
	\ee
	which holds due to the symmetry of $\eta_t$.
	In the last equality of \eqref{boundeqs1}, we have used similar arguments in connection with the fact $x_{t+1}=(A+BK)x_t+\omega_t+B\eta_t$.
	Taking the sum $\sum\limits_{t=0}^{\infty}(\cdot)$, applying $\mathbb{E}_{\rho_K}[\cdot]$ to both sides of \eqref{boundeqs1} and 
	multiplying by $\frac{1}{\sigma^2}$ yields $G(K)=\nabla J(K)$, see \eqref{exgradient2}.

	Since $\hat G^{(L)}(K)$ is the estimator of $\nabla J(K)$, we focus on the bound of the bias and split the bias into two parts
	\be\label{boundeq2}
	\begin{aligned}
		&\nabla J(K)-\mathbb{E}[\hat G^{(L)}(K)]\\
		=&\frac{1}{\sigma^2}\mathbb{E}\left[\sum_{t=L+1}^\infty(u_t-Kx_t)x_t^\top \sum_{k=t}^\infty\gamma^kc_k\right]+\frac{1}{\sigma^2}\mathbb{E}\left[\sum_{t=0}^L(u_t-Kx_t)x_t^\top \sum_{k=L+1}^\infty \gamma^kc_k\right],
	\end{aligned}
	\ee 
    where we have used that $G(K)=\nabla J(K)$. In the following, we will further simplify and estimate the above two terms, respectively. 
	To this end, we show for each $t\leq L$ that
	\be\label{boundeq1}
		\mathbb{E}\left[(u_t-Kx_t)x_t^\top \sum_{k=L+1}^\infty \gamma^kc_k\Big|\mathcal F_{t}\right]\\
		=2\sigma^2\gamma^{L+1}B^\top \mathcal T ^{L-t}(P_{K,\gamma})(A+BK)x_tx_t^\top .
	\ee
	For $t<L$, we observe that
		\be\label{boundeq11}
	\begin{aligned}
		&\mathbb{E}\left[(u_t-Kx_t)x_t^\top \sum_{k=L+1}^\infty \gamma^kc_k\Big|\mathcal F_{t}\right]\\
		=&\mathbb{E}\left[\mathbb{E}\left[\eta_tx_t^\top \sum_{k=L+1}^\infty \gamma^kc_k\Big|\mathcal F_{L+1}\right]\Big|\mathcal F_{t}\right]\\
		=&\gamma^{L+1}\mathbb{E}\left[\eta_tx_t^\top V_{K}(x_{L+1})\Big|\mathcal F_t\right]\\
		=&\gamma^{L+1}\mathbb{E}\left[\eta_t[(A+BK)x_L+\widetilde\omega_L]^\top P_{K,\gamma}[(A+BK)x_{L}+\widetilde\omega_L]\Big|\mathcal F_t\right]x_t^\top \\
		=&\gamma^{L+1}\mathbb{E}\left[\eta_tx_{L}^\top \mathcal T (P_{K,\gamma})x_{L}\Big|\mathcal F_t\right]x_t^\top (since\;\widetilde\omega_L\;is\;independent\;from\; \eta_t)\\
		&\cdots\\
		=&\gamma^{L+1}\mathbb{E}\left\{\eta_tx_{t+1}^\top \mathcal T ^{L-t}(P_{K,\gamma})x_{t+1}\Big|\mathcal F_t\right\}x_t^\top \\
		=&2\sigma^2\gamma^{L+1}B^\top \mathcal T ^{L-t}(P_{K,\gamma})(A+BK)x_tx_t^\top ,
	\end{aligned}
	\ee
	and for $t=L$ it holds
		\be\label{boundeq12}
	\begin{aligned}
		&\mathbb{E}\left[(u_t-Kx_t)x_t^\top \sum_{k=t+1}^\infty \gamma^kc_k\Big|\mathcal F_{t}\right]\\
		=&\mathbb{E}\left[\mathbb{E}\left[\eta_tx_t^\top \sum_{k=t+1}^\infty \gamma^kc_k\Big|\mathcal F_{t+1}\right]\Big|\mathcal F_{t}\right]\\
		=&\gamma^{t+1}\mathbb{E}\left[\eta_tx_t^\top V_{K}(x_{t+1})\Big|\mathcal F_t\right]\\
		=&\gamma^{t+1}\mathbb{E}\left[\eta_t[(A+BK)x_t+\widetilde\omega_t]^\top P_{K,\gamma}[(A+BK)x_{t}+\widetilde\omega_t]\Big|\mathcal F_t\right]x_t^\top \\
		=&2\sigma^2\gamma^{t+1}B^\top P_{K,\gamma}(A+BK)x_tx_t^\top .
	\end{aligned}
	\ee

	In \eqref{boundeq11} and \eqref{boundeq12}, similar arguments as in \eqref{boundeqs1} are used, in particular $\mathbb{E}[\eta_t]=0$. Hence, \eqref{boundeq1} holds.

	Since $K$ is stable, \eqref{eq:stability} yields 
	\be\label{eq:exp_xtxt}
	\left\Vert  \mathbb{E}[x_{t}x_{t}^\top ]\right\Vert_2\leq {\Gamma_K^2}\sum_{k=0}^t\rho^{2k}\Vert D_{\widetilde \omega}\Vert_2
	\ee
	for some constant $0<\rho<1$. Using this, \eqref{eq:J_est_1} and \eqref{boundeq1}, we can estimate the second term in \eqref{boundeq2} as follows
	\be\label{lemmaeq4}
	\begin{aligned}
		&\left\Vert\frac{1}{\sigma^2}\mathbb{E}\left[\sum_{t=0}^L(u_t-Kx_t)x_t^\top \sum_{k=L+1}^\infty \gamma^kc_k\right]\right\Vert_2\\
		\leq& {2\gamma^{L+1}}\Vert B\Vert_2\sum_{t=0}^L\left  \Vert P_{K,\gamma}\Vert_2 \Gamma_K^{3}\rho^{2L-2t+1}\Vert\mathbb{E}[x_tx_t^\top ]\right\Vert_2\\
		\leq& 2\gamma^{L+1}\Vert B\Vert_2\frac{(1-\gamma)}{\gamma\sigma_{\mathrm{min}}(D_{\widetilde \omega})}J(K)\frac{\Gamma_K^5}{(1-\rho^2)^2}\Vert D_{\widetilde \omega}\Vert_2.
	\end{aligned}
	\ee
	
	By using \eqref{boundeqs1}, \eqref{eq:exp_xtxt} and \eqref{eq:J_est_1}, the first term in \eqref{boundeq2} can be estimated as:
	\be\label{lemmaeq3}
	\begin{aligned}
		&\left\Vert\frac{1}{\sigma^2}\mathbb{E}\left[\sum_{t=L+1}^\infty(u_t-Kx_t)x_t^\top \sum_{k=t}^\infty\gamma^kc_k\right]\right\Vert_2\\
		\leq &2\Big(\Vert R\Vert_2\Vert K\Vert_2+\Gamma_K\Vert B\Vert_2\Vert P_{K,\gamma}\Vert_2\Big)\sum_{t=L+1}^\infty \frac{\Gamma_K^2}{1-\rho^2}\gamma^t\Vert D_{\widetilde \omega}\Vert_2\\
		\leq&2\Big(\Vert R\Vert_2\Vert K\Vert_2+\Gamma_K\Vert
        B\Vert_2\frac{(1-\gamma)}{\gamma\sigma_{\mathrm{min}}(D_{\widetilde
    \omega})}J(K)\Big)\Vert D_{\widetilde
    \omega}\Vert_2\frac{\gamma^{L+1}\Gamma_K^2}{(1-\rho^2)(1-\gamma)}.
	\end{aligned}
	\ee
	Combining \eqref{lemmaeq4} and \eqref{lemmaeq3} yields \eqref{lemmaeq5} with
	\[
	C_2=4\Vert D_{\widetilde \omega}\Vert_2
	\cdot\max\left\{\Gamma_K\Vert B\Vert_2
	\frac{(1-\gamma)J(K)}{\gamma\sigma_{\mathrm{min}}(D_{\widetilde \omega})},
	\Vert R\Vert_2\Vert K\Vert_2\right\}.
	\]

	Finally, we derive a bound for the variance of $\hat G^{(L)}$.
	\be\label{boundkeypg}
	\begin{aligned}
		\mathbb E\left\Vert \hat G^{(L)}(K)\right\Vert^2_F=& \frac{1}{\sigma^4}\mathbb E\left\Vert \sum_{t=0}^L(u_t-Kx_t)x_t^\top \sum_{k=t}^L\gamma^kc_k \right\Vert_F^2\\
		\leq&  \frac{L}{\sigma^4}\sum_{t=1}^L\mathbb E\left[\|x_{t}\|^2_2\|\eta_{t}\|_2^2\Big(\sum_{k=t}^L\gamma^kc_k\Big)^2\right] \\
		\leq& \frac{L}{\sigma^4}\sum_{t=1}^L(L-t+1)\mathbb E\left[\|x_{t}\|^2_2\|\eta_{t}\|_2^2\sum_{k=t}^L\gamma^{2k}c^2_k\right] \\
		\leq& \frac{L^2}{\sigma^4}\sum_{t=1}^L\gamma^{2t}\sum_{k=t}^L\mathbb E\left[\|x_{t}\|^2_2\|\eta_{t}\|_2^2c^2_k\right].
	\end{aligned}
	\ee
	Hence, we only need to estimate each term of \eqref{boundkeypg}:
	\be\label{boundkeypg1}
	\mathbb E\left[\|x_{t}\|^2_2\|\eta_{t}\|_2^2c^2_k\right]\leq O\left(\|K\|_2^4+\|S\|_2^2\right)\mathbb E\Big[\|x_k\|^4_2\|x_t\|_2^2\|\eta_t\|_2^2\Big]+O\left(\|R\|_2^2\right)\Tr[D_K]\sigma^6.
	\ee
	Let $\Sigma = \sum_{i=t}^{k-1}(A+BK)^{k-1-i}\widetilde\omega_t$. Then $x_k=(A+BK)^{k-t}x_t+\Sigma$ and $\Sigma$ is independent from $x_t$. 
	We can estimate the bound of $\mathbb E\Big[\|x_k\|^4_2\|x_t\|_2^2\|\eta_t\|_2^2\Big]$:
	\be\label{boundkeypg2}
	\begin{aligned}
		\mathbb E\Big[\|x_k\|^4_2\|x_t\|_2^2\|\eta_t\|_2^2\Big]=&\mathbb E\Big[\big\|(A+BK)^{k-t}x_t+\Sigma\big\|^4_2\|x_t\|_2^2\|\eta_t\|_2^2\Big]\\
		\leq&\mathbb E\Big\{\big[2\|(A+BK)^{k-t}x_t\|_2^2+2\|\Sigma\|_2^2\big]^2\|x_t\|_2^2\|\eta_t\|_2^2\Big\}\\
		\leq &8\mathbb E\Big\{\big[\|(A+BK)^{k-t}x_t\|_2^4+\|\Sigma\|_2^4\big]\|x_t\|_2^2\|\eta_t\|^2_2\Big\}\\
		\leq &8d\sigma^2\mathbb E\big[\|(A+BK)^{k-t}x_t\|_2^4\|x_t\|_2^2\big]+8\mathbb E\|x_t\|^2_2\mathbb E\big(\|\Sigma\|^4_2\|\eta_t\|_2^2\big)\\
		\leq &O\Big(\big[\mathbb E_{\mu_K}\|x\|^6+\Tr[D_K]\mathbb
        E_{\mu_K}\|x\|^4\big]\sigma^2+ \Tr[D_K]\sigma^6\Big).
	\end{aligned}
	\ee
	We observe that \eqref{boundkeypg2} is bounded yielding that \eqref{boundkeypg1} is bounded by a constant $C>0$. Inserting this in \eqref{boundkeypg}, we get the estimation
	\be\label{boundkeypg3}
	\begin{aligned}
		\mathbb E\left\Vert \hat G^{(L)}(K)\right\Vert^2_F
		\leq \frac{CL^3}{\sigma^4(1-\gamma^2)}
	\end{aligned}
	\ee
	such that \eqref{lemmaeq5} holds with $C_3=\frac{C}{\sigma^4}$.
\end{proof}

\begin{assumption}\label{assumption:GLK_bounded}
	There is a positive number $M_G$ such that $\big\|\hat G^{L}(K)\big\|_F\leq M_G$ for any stable $K$. 
\end{assumption}

At the end of this section, we present a theorem that guarantees the convergence of \cref{samplegrad}. In order to prove this convergence, we need to assume that the error tolerance $\epsilon$ and the confidence $\delta$ is small enough such that following inequality holds:
\be\label{bounderr_con}
\delta\epsilon<\min\left(1,\frac{10(J(K^{(0)})-J(K^*))}{3\log [120 (J(K^{(0)})-J(K^*))]}\right).
\ee
\begin{theorem}\label{PGTH}
Let \cref{assumption:GLK_bounded} hold,
	$K^{(0)}$ be stable and suppose that $\gamma$ is sufficiently close to $1$ and $0<\beta<1$ is small enough such that \eqref{eq:lemma_restriction_to_dom2} is satisfied with $\nu=\frac{10}{\delta}$ and $K=K^{(0)}$. For any error tolerance $\epsilon$ and confidence $\delta$  satisfying \eqref{bounderr_con}, suppose that the sample size $L$ is large enough such that \[\gamma^{L+1}\left(\frac{\Gamma^2}{1-\rho^2}+\frac{1}{1-\gamma}\right)\leq O(\sqrt{\delta\epsilon/n})\]
	with $\Gamma=\Gamma_{K^{(0)},\nu}$ and the step size $\alpha$ is chosen to satisfy $\alpha\leq O\left(\frac{\delta\epsilon(1-\gamma^2)\sigma^2}{L^3d}\right)$. After $T$ iterations with $T=O( \frac{1}{\alpha}\log[\frac{ J(K^{(0)})-J(K^*)}{\delta\epsilon}])$, \cref{samplegrad}  yields an iterate $K^{(T)}$ such that
	\[
	J(K^{(T)})-J( K^*)\leq \epsilon
	\]
	holds with probability greater than $1-\delta$. 
\end{theorem}

\begin{proof}
	For any $t=0,1,\cdots$, define the error $\Delta_t=J(K^{(t)})-J(K^*)$ and the stopping time 
	\[
	\tau:=\min \left\{t | \Delta_{t}>10 \Delta_{0}/\delta\right\}.
	\]
	We first note that \cref{lem:restriction_to_dom} yields $K^{(t)}\in$\textbf{Dom}$_{K^{(0)},\nu}$ for all $t<\tau$ since $\gamma$ and $\beta$ is chosen such that \eqref{eq:constant_epsilon_K} and \eqref{eq:lemma_restriction_to_dom2} hold with $\nu=\frac{10}{\delta}$ and $K=K^{(0)}$. Hence $C_2$ and $C_3$ are both uniform bounded in \textbf{Dom}$_{K^{(0)},\nu}$.
	
	Next, for simplicity, we define $\mathbb{E}^{t}:=\mathbb{E}\left[\cdot | \mathcal{A}_{t}\right]$ as the expectation operator conditioned on the sigma field $\mathcal{A}_{t}$ which contains all the randomness of the first $t$ iterates. 
	Since the  gradient $\nabla J$ is locally Lipschitz, which is shown in \cite{malik2018derivative}, there are a uniform Lipschitz constant $\phi_{0}$ and  a uniform radius $\rho_0$ such that
	\be
	J(K')-J(K)\leq \left\langle\nabla J(K), K'-K\right\rangle+\frac{\phi_0}2\left\|K'-K\right\|_F^2
	\ee
	for any $K$ and $K'$ with $\|K-K'\|_F\leq \rho_0$. We choose $\alpha$ sufficiently small such that $\alpha M_G\leq\rho_0$.   
	
	Let L be sufficiently large such  that $C_2\frac{\Gamma^2}{1-\rho^2}\sqrt{n}\gamma^{L+1}(\frac{\Gamma^2}{1-\rho^2}+\frac{1}{1-\gamma})\leq \sqrt{\delta\epsilon\mu/180}$. Using this and in particular \cref{bounded}, we obtain 
	\be\label{eq:convergence_policy_gradient_inequality}
	\begin{aligned}
		&\mathbb{E}^{t}\left[J\left(K^{(t+1)}\right)-J\left(K^{(t)}\right)\right]\\
		\leq& \mathbb{E}^{t}\left[\left\langle\nabla J\left(K^{(t)}\right), K^{(t+1)}-K^{(t)}\right\rangle+\frac{\phi_{0}}{2}\left\|K^{(t+1)}-K^{(t)}\right\|_{F}^{2}\right]\\
		=&-\alpha\left\langle \nabla J\left(K^{(t)}\right), \nabla J\left(K^{(t)}\right)\right\rangle+\frac{\phi_{0} \alpha^{2}}{2} \mathbb{E}^{t}\left[\left\|\hat G^{(L)}\left(K^{(t)}\right)\right\|_F^{2}\right]\\
		&+\alpha\left\langle \nabla J\left(K^{(t)}\right),\nabla J\left(K^{(t)}\right)-\mathbb{E}^{t}[\hat G^{(L)}(K^{(t)})]\right\rangle\\
		\leq&-\alpha\Big\|\nabla J\left(K^{(t)}\right)\Big\|_F^2+\frac{\phi_{0} \alpha^{2}}{2} \mathbb{E}^{t}\left[\left\|\hat G^{(L)}\left(K^{(t)}\right)\right\|_F^{2}\right]\\
		&+\alpha\Big\|\nabla J\left(K^{(t)}\right)\Big\|_F\Big\|\nabla J\left(K^{(t)}\right)-\mathbb{E}^{t}[\hat G^{(L)}(K^{(t)})]\Big\|_F\\
		\leq&-\alpha\Big\|\nabla J\left(K^{(t)}\right)\Big\|_F^2+\frac{\phi_{0} \alpha^{2}}{2}\left[\mathrm{Var}\left(\hat G^{(L)}(K^{(t)})\right)+\left\|\mathbb E^t[\hat G^{(L)}(K^{(t)})]\right\|_F^2\right]\\
		&+\alpha\Big\|\nabla J\left(K^{(t)}\right)\Big\|_F\sqrt{\delta\epsilon\mu/180}\\
		\leq& -\frac34\alpha\Big\|\nabla J\left(K^{(t)}\right)\Big\|_F^2+\frac{\alpha\delta\epsilon\mu}{180}+\frac{\phi_{0} \alpha^{2}}{2}G_2+\phi_{0} \alpha^{2}\left(\Big\|\nabla J\left(K^{(t)}\right)\Big\|_F^2+\frac{\delta\epsilon\mu}{180}\right),
	\end{aligned}
	\ee
	where $G_2=\mathrm{Var}\left(\hat G^{(L)}(K^{(t)})\right)=\mathbb{E}^{t}\left[\left\|\hat G^{(L)}\left(K^{(t)}\right)\right\|_F^{2}\right]-\left\|\mathbb{E}^{t}[\hat G^{(L)}\left(K^{(t)}\right)\right\|_F^{2}$. We note that in the last estimation of \eqref{eq:convergence_policy_gradient_inequality} the inequality 
	\begin{align*}
	\alpha\Big\|\nabla J\left(K^{(t)}\right)\Big\|_F\sqrt{\delta\epsilon\mu/180}\leq \frac14\alpha\Big\|\nabla J\left(K^{(t)}\right)\Big\|_F^2+\frac{\alpha\delta\epsilon\mu}{180}
	\end{align*}
	is used. We assume that $\alpha\leq \frac1{2\phi_0}$. By the PL condition \eqref{PLcond}, we have
	\be
	\begin{aligned}
		\mathbb{E}^{t}\left[\Delta_{t+1}-\Delta_t\right]\leq&-\frac14\alpha\Big\|\nabla J\left(K^{(t)}\right)\Big\|_F^2+\frac{\phi_{0} \alpha^{2}}{2}G_2+\frac{\alpha\delta\epsilon\mu}{120}\\
		\leq &-\frac14\alpha\mu\Delta_t+\frac{\phi_{0} \alpha^{2}}{2}G_2+\frac{\alpha\delta\epsilon\mu}{120}.
	\end{aligned}
	\ee
	Applying successively this inequality, we obtain a similar result as in \cite{malik2018derivative}:
	\[
	\begin{aligned}
	\mathbb{E}\left[\Delta_{t+1} 1_{\tau>t+1}\right] \leq&\left(1-\frac{\alpha \mu}{4}\right)^{t+1} \Delta_{0}+(\frac{\phi_{0} \alpha^{2}}{2}G_2+\frac{\alpha\mu\epsilon\delta}{120}) \sum_{i=0}^{t}\left(1-\frac{\alpha \mu}{4}\right)^{i}\\
	\leq&\left(1-\frac{\alpha \mu}{4}\right)^{t+1} \Delta_{0}+\frac{2}{\mu}\alpha\phi_0G_2+\frac{\epsilon\delta}{30},\\
	\end{aligned}
	\]
	where we have used that $\mathbb{E}^{t}\left[\Delta_{t}\right]=\Delta_{t}$. By \eqref{lemmaeq6},  we observe that taking $\alpha\leq\frac{\mu\epsilon\delta}{240\phi_0  C_3\frac{L^3}{(1-\gamma^2)}}$ implies $\frac{2}{\mu}\alpha\phi_0G_2\leq \frac{\epsilon\delta}{120}$. We note that this condition on $\alpha$ as well as $\alpha M_G\leq\rho_0$ and $\alpha\leq \frac1{2\phi_0}$ are satisfied for $\alpha\leq O\left(\frac{\delta\epsilon(1-\gamma^2)\sigma^2}{L^3d}\right)$.
	Setting $t=T-1$, we observe that for  
	\[
	T\geq C\cdot\frac{1}{\alpha}\log\left[\frac{120[ J(K^{(0)})-J(K^*)]}{\delta\epsilon}\right]
	\]
	with a sufficiently large constant $C$ the inequality 
	$\mathbb{E}^{t}\left[\Delta_{t+1} 1_{\tau>t+1}\right]\leq\frac{\epsilon\delta}{5}$ holds.
	
	By using the same techniques as in proof of the Proposition 1 in \cite{malik2018derivative}, we observe that $\mathbb{P}\left\{1_{\tau \leq T}\right\}\leq 4\delta/5$. By Chebyshev's inequality, we have
	\[
	\begin{aligned}
	\mathbb{P}\left\{\Delta_{T} \geq \epsilon\right\} \leq& \mathbb{P}\left\{\Delta_{T} 1_{\tau>T} \geq \epsilon\right\}+\mathbb{P}\left\{1_{\tau \leq T}\right\}\\
	\leq& \frac{1}{\epsilon} \mathbb{E}\left[\Delta_{T} 1_{\tau>T}\right]+\mathbb{P}\left\{1_{\tau \leq T}\right\}\\
	\leq& \frac{\delta}{5}+\frac{4\delta}{5}\\
	=&\delta.
	\end{aligned}
	\]
	This completes the proof.
\end{proof}

\section{The Actor-Critic Algorithm}\label{sect:AC}
In the policy gradient algorithm, we update the policy through pure sampling updates. Hence, the policy gradient has high variance, which slows down the speed of convergence. 
A popular method to reduce the variance  is the Actor-Critic algorithm, which replaces the Monte Carlo method by the bootstrapping method. The policy gradient in the Actor-Critic algorithm has the following form:
\be\label{sgb2}
\hat G_{AC}^{(L)}(K)=\frac{1}{\sigma^2}\sum_{t=0}^L\left[(u_t-Kx_t)x_t^\top \gamma^t\tilde\delta_t\right].
\ee 
where $\tilde\delta_t:=c_t+\gamma\widetilde V(x_{t+1})-\widetilde V(x_{t})$. 
We investigate the bias and the variance of the estimators.

\begin{lemma}\label{lemmaacb}
	Suppose that the policy $K$ is stable and $\widetilde V$ is an $\epsilon_0-$approximation of the state value function. Then for the estimation $\hat G_{AC}^{(L)}(K)$ shown in \eqref{sgb2} the following inequalities hold 
	\begin{eqnarray}
	\label{bound61}\left\Vert \nabla J(K)-\mathbb{E}[\hat G_{AC}^{(L)}(K)]\right\Vert_2&\leq& C_2\gamma^{L+1}\frac{\Gamma_K^2}{(1-\rho^2)(1-\gamma)}+C_4\Gamma_K J(K)\epsilon_0,\\
	\label{eq:ACSMB}\mathbb{E}\left\Vert \hat G_{AC}^{(L)}(K)\right\Vert^2_F&\leq& \left[ \frac{C_5}{\sigma^2}+O(\Tr[D_K]\sigma^2)\right]\frac{L}{1-\gamma^2},
	\end{eqnarray}
	where $C_2$ is given as in \cref{bounded} and $C_4$, $C_5$ depend on $\|A\|_2$, $\|B\|_2$, $\|S\|_2$, $\|R\|_2$, $\|K\|^2_2$, $\Gamma_K$, $\frac1{1-\rho^2}$, $\widetilde\omega$ and $(1-\gamma)J(K)$ with $\rho\in (\rho(A+BK),1)$. 
\end{lemma}
\begin{proof}
	Analogous to the proof of \cref{bounded}, can split the bias of $\hat G_{AC}^{(L)}$ into two parts:
	\be\label{boundeq61}
	\begin{aligned}
		\nabla J(K)-\mathbb{E}[\hat G_{AC}^{(L)}(K)]=&\frac{1}{\sigma^2}\mathbb{E}\left[\sum_{t=L+1}^\infty(u_t-Kx_t)x_t^\top \sum_{k=t}^\infty\gamma^kc_k\right]\\
		&+\frac{1}{\sigma^2}\mathbb{E}\left[\sum_{t=0}^L\gamma^{t+1}(u_t-Kx_t)x_t^\top(V_K(x_{t+1})-\widetilde V(x_{t+1}))\right].
	\end{aligned}
	\ee 
	
	We have discussed the first term in the proof of \cref{bounded} and its upper bound is $ C_2\frac{\Gamma_K^2}{(1-\rho^2)(1-\gamma)}\gamma^{L+1}$.
	Let $V_K(x)=\phi(x)^\top\theta^*=x^\top \Theta_K^* x+\theta^*_0$ and  $\widetilde V(x)=\phi(x)^\top \theta=x^\top \Theta x+\theta_0$.
	We observe that the second term of \eqref{boundeq61} has the following equivalent form:
	\be\label{boundeq62}
	\begin{aligned}
		&\frac{1}{\sigma^2}\mathbb{E}\left[\sum_{t=0}^L\gamma^{t+1}(u_t-Kx_t)x_t^\top(V_K(x_{t+1})-\widetilde V(x_{t+1}))\right]\\
		=&\frac{1}{\sigma^2}\mathbb{E}\left[\sum_{t=0}^L\gamma^{t+1}\eta_t(\Tr[(\Theta^*_{K}-\Theta)x_{t+1}x_{t+1}^\top])x_t^\top\right]\\
		=&\frac{1}{\sigma^2}\mathbb{E}\left[\sum_{t=0}^L\gamma^{t+1}\eta_t[\eta_t^\top B^\top(\Theta^*_{K}-\Theta)(A+BK)x_t]x_t^\top\right]\\
		=&2B^\top(\Theta^*_{K}-\Theta)(A+BK)\mathbb{E}\left[\sum_{t=0}^L\gamma^{t+1}x_tx_t^\top\right].
	\end{aligned}
	\ee 
	Since $\left\|\mathbb{E}\left[\sum_{t=0}^L\gamma^{t+1}x_tx_t^\top\right]\right\|_2\leq \|\Sigma_{K,\gamma}\|_2\leq \frac{J(K)}{\sigma_{\mathrm{min}}(S)}$ and $\widetilde V$ is an $\epsilon_0$-approximation of $V_K$, the right hand of \eqref{boundeq62} is bounded by $\frac{\|B\|_2\Gamma_K}{\sigma_{\mathrm{min}}(S)} J(K)\epsilon_0$, which yields \eqref{bound61}.

	%
	Since $\|\theta_1\|_2\leq \|\theta_1^\star\|_2+\epsilon_0\leq \frac{(1-\gamma)J(K)}{\gamma\sigma_{\mathrm{min}}(D_{\widetilde \omega})}+\epsilon_0$ by \eqref{eq:J_est_1} and $|\theta_0|\leq |\theta_0^*|+\epsilon_0=\frac{\gamma}{1-\gamma}\Tr[P_{K,\gamma}D_{\widetilde\omega}]+\frac{\sigma^2\Tr[ R]}{1-\gamma}+\epsilon_0=J(K)+\epsilon_0$, the following inequality holds:
	\be
	\begin{aligned}
		\tilde \delta_t^2&\leq 3\left[c_t^2+(1-\gamma)^2|\theta_0|^2+\|x_t^\top x_t-\gamma x_{t+1}^\top x_{t+1}\|^2_F\|\theta_1\|_2^2\right]\\
        &\leq 3c_t^2+O((1-\gamma)^2J^2(K))(1+\|x_t\|_2^4+ \|x_{t+1}\|^4_2).
	\end{aligned}
	\ee
	Hence, we obtain the inequality \eqref{eq:ACSMB}:
	\be
	\begin{aligned}
		\mathbb{E}\left\Vert \hat G_{AC}^{(L)}(K)\right\Vert^2_F
		&\leq \frac{L}{\sigma^4}\sum_{t=1}^L\gamma^{2t}\mathbb E\left(\|x_{t}\|_2^2\|\eta_{t}\|_2^2|\tilde\delta_{t}|^2\right)\\
		&\leq\frac{L}{\sigma^4}\sum_{t=1}^L\gamma^{2t}\left[C_5\sigma^2+O(\Tr[D_K])\sigma^6\right]\\
		&\leq\left[ \frac{C_5}{\sigma^2}+O(\Tr[D_K]\sigma^2)\right]\frac{L}{1-\gamma^2},
	\end{aligned}
	\ee
	where the third inequality holds by the same skill in \eqref{sigmahb} and 
	\[C_5=O\left([(1-\gamma)^2J^2(K)+\|K\|^4_2]\mathbb E_{\mu_K}\|x\|^6\right).\]
	In the estimations above similar techniques as in \eqref{boundkeypg} are used.
	
\end{proof}

Now we can obtain a convergence result for the AC algorithm which is similar to the one in
\cref{PGTH}. In order to prove this result, we need the following assumption.
\begin{assumption}\label{assumption:GLACK_bounded}
There is a positive number $M^{AC}_G$ such that $\big\|\hat G_{AC}^{(L)}(K)\|_F\leq M^{AC}_G$ for any stable $K$.
\end{assumption}
\begin{theorem}\label{thm:policy_ac_convergence}
Suppose that \cref{assumption:GLACK_bounded} is satisfied and let $K^{(0)}$ be a stable policy. Moreover, suppose that $\gamma$ is sufficiently close to $1$ and $0<\beta<1$ small enough such that \eqref{eq:lemma_restriction_to_dom2} holds with $\nu=\frac{10}{\delta}$.
For any error tolerance $\epsilon$ and confidence $\delta$  satisfying \eqref{bounderr_con}, suppose that the sample size $L$ is large enough and the error of the approximation value function $\epsilon_0$ is small enough such that \[\gamma^{L+1}\frac{1}{1-\gamma}\leq O(\sqrt{\delta\epsilon/n})\quad \text{and} \quad J(K^{(0)})\epsilon_0\leq O(\sqrt{\delta\epsilon/n}).\]
The step size $\alpha$ is chosen such that $\alpha\leq O\left(\frac{\delta\epsilon(1-\gamma^2)}{Ld}\right)$ holds. After $T$ iterations with $T=O( \frac{1}{\alpha}\log[\frac{ J(K^{(0)})-J(K^*)}{\delta\epsilon}])$ iterations, the iterate $K^{(T)}$ satisfies
\[
J(K^{(T)})-J( K^*)\leq \epsilon
\]
with probability greater than $1-\delta$.
\end{theorem}

\begin{proof}
The proof is similar to the proof of \cref{PGTH}. 
\end{proof}

Finally, we analyze the sample complexity of the policy gradient method and the
AC algorithm. We assume that the discount factor $\gamma$ with $0<\gamma<1$ is
close to $1$ such that $\log(\gamma)\approx\gamma-1$. By the definition of
$J(K)$ and in particular by its representation in
\eqref{eq:cost_function_simplified}, we obtain
$J(K^{(0)})=O(\frac{1}{1-\gamma})$. For the AC algorithm, we have to require
that $J(K^{(0)})\epsilon_0=O(\sqrt{\delta\epsilon/n})$. Then the TD-learning
algorithm needs $O\left(\frac{1}{\delta\epsilon(1-\gamma)^4}\right)$ steps by
\cref{thm:semi-gradient_update}. However, the sample size $L$ in the AC algorithm is only $O\left(\frac{\log\frac{n}{(1-\gamma)\delta\epsilon}}{\log(\gamma)}\right)\approx O((1-\gamma)^{-2})$. We can sample $\frac{1}{(1-\gamma)^2}$ trajectories parallelly. Then the variance of the gradient $\hat G^{L}(K)$ becomes $(1-\gamma)L=O(\frac{1}{1-\gamma})$. Using similar arguments as in the proof of \cref{PGTH}, one can prove that the iteration time $T$ of the AC algorithm is equal to $O(\frac{1}{\delta\epsilon(1-\gamma)})$. From the statements shown above, we conclude the complexities given in \cref{tab:complexity}.

\section{Conclusion}\label{sect:conclusion}
Reinforcement learning has achieved success in many fields but lacks theoretical understanding in the continuous case. In this paper, we apply well known algorithms in RL to a basic model LQR. First, we show the convergence of the policy iteration with TD learning, which is hard to prove in general cases. Then we obtain the linear convergence of the policy gradient method and AC algorithm. Finally, we compare the sample complexity of those algorithms. 

The results of this paper are proved for the LQR setting, which allows us to
restrict ourselves to linear policies. For extensions to more general problems,
this restriction to a linear framework is not possible anymore. Consequently,
the policy function may depend nonlinearly on its parameters which makes in
particular the convergence analysis of optimization methods such as the policy gradient method much more involved. A further difficulty is that the PL condition, which guarantees that stationary policies are globally optimal, may not hold anymore. However, since the LQR can be used as approximation for more general nonlinear problems, the techniques developed in this paper can serve as important tool for the treatment of more general problems.

\section*{Acknowledgments.}
The research is supported in part by the NSFC grant 11831002 and  Beijing Academy  of Artificial Intelligence.

\bibliographystyle{siamplain}
\bibliography{mybib}
\end{document}